\documentclass[11pt, a4paper]{amsart}
\usepackage{url}

\usepackage[colorlinks,citecolor=cyan,linkcolor=magenta]{hyperref}
\usepackage[nameinlink]{cleveref}
\usepackage{multirow}
\usepackage{subfig}

\usepackage{amsmath}
\usepackage{amssymb}
\usepackage{graphicx}
\usepackage[dvipsnames]{xcolor}
\usepackage{blkarray}
\usepackage{hyperref}

\newtheorem{thm}{Theorem}
\newtheorem*{que*}{Question}
\newtheorem{prop}[thm]{Proposition}
\newtheorem{lem}[thm]{Lemma}

\newtheorem{remark}[thm]{Remark}

\newtheorem*{nonumberthm}{Theorem}

\begin{document}

\title{Trace field degrees in the Torelli group}

\author{Erwan Lanneau}
\address{
UMR CNRS 5582, 
Univ. Grenoble Alpes, CNRS, Institut Fourier, F-38000 Grenoble, France}
\email{erwan.lanneau@univ-grenoble-alpes.fr}

\author{Livio Liechti}
\address{Department of Mathematics\\
University of Fribourg\\
Chemin du Mus\'ee 23\\
1700 Fribourg\\Switzerland}
\email{livio.liechti@unifr.fr}

\begin{abstract}
We show that for~$g\ge 2$, all integers~$1\le d\le 3g-3$ arise 
as trace field degrees of pseudo-Anosov mapping classes in the Torelli group
of the closed orientable surface of genus~$g$. 
Our method uses the Thurston--Veech construction of pseudo-Anosov maps, 
and we provide examples where the stretch factor has algebraic degree any even 
number between two and~$6g-6$. 
This validates a claim by Thurston from the 1980s. 
\end{abstract}

\maketitle

\section{Introduction}

\noindent
Let~$S_g$ be the closed orientable surface of genus~$g\ge 2$.
A homeomorphism~$f$ of~$S_g$ is \emph{pseudo-Anosov} if there exists a pair of transverse 
singular measured foliations~$\mathcal F^u$ and~$\mathcal F^s$ and a real 
number~$\lambda >1$ such that~$f(\mathcal F^u)=\lambda \mathcal F^u$ 
and~$f(\mathcal F^s)=\lambda^{-1} \mathcal F^s$.
The number~$\lambda > 1$ is the \emph{stretch factor} of the 
pseudo-Anosov map, and is an algebraic integer. 
The degree of the field extension~$\mathbb{Q}(\lambda) : \mathbb{Q}$ is the \emph{stretch 
factor degree}, and is bounded from above by the dimension of the Teichm\"uller space 
of~$S_g$, namely~$6g-6$~\cite{Th}. 
\smallskip

\noindent
There is another field which plays a central role in this article:~$\mathbb{Q}(\lambda + \lambda^{-1})$. 
This field is the \emph{trace field} and is uniquely determined by the 
pair $(\mathcal F^u,\mathcal F^s)$~\cite{Kenyon:Smillie, Gutkin:Judge}.
The degree of the field extension~$\mathbb{Q}(\lambda + \lambda^{-1}) : \mathbb{Q}$ is 
the~\emph{trace field degree}, and is bounded from above by~$3g-3$. 
\smallskip

\noindent
Strenner showed that for~$S_g$, all integers~$1\le d \le 3g-3$ 
arise as trace field degrees of pseudo-Anosov maps~\cite{Strenner:algebraic}.
Furthermore, Strenner determined the set of integers which arise as stretch factor 
degrees: all even integers between~$2$ and~$\le 6g-6$, as well as all odd  
integers between~$3$ and~$3g-3$~\cite{Strenner:algebraic},
{but none of Strenner's examples are elements of the Torelli group (see~\Cref{sec:proof:strategy}).}

\subsection{Torelli groups} 
The \emph{Torelli group}~$\mathcal{I}(S_g)$ is the kernel of the symplectic representation 
of the mapping class group~$\mathrm{Mod}(S_g)$. In~\cite[Problem 10.6]{Margalit:question}, 
Margalit asked which stretch factor degrees arise for pseudo-Anosov mapping classes in 
the Torelli group. 
Our first main result completely answers the question of trace field degrees arising in Torelli groups. 

\begin{thm}
\label{thm:trace}
Let~$g\ge 2$. Every integer~$1\le d\le 3g-3$ arises as the trace field degrees of a pseudo-Anosov 
mapping class in the Torelli group~$\mathcal{I}(S_g)$.
\end{thm}

\noindent
The stretch factor~$\lambda$ satisfies the quadratic equation~$t^2-(\lambda+\lambda^{-1})t + 1$ 
over the trace field~$\mathbb{Q}(\lambda+ \lambda^{-1})$. Hence, the field 
extension~$\mathbb{Q}(\lambda) : \mathbb{Q}(\lambda+ \lambda^{-1})$ has degree one or two. 
We obtain our second main result by showing that for all trace field degrees there exist instances 
where this field extension degree equals two. 

\begin{thm}
\label{thm:stretch}
Let~$g\ge 2$. Every even integer $2d$ {with} $2 \le 2d \le 6g-6$ arises as the stretch factor degree of a pseudo-Anosov 
mapping class in the Torelli group~$\mathcal{I}(S_g)$.
\end{thm}

\subsection{The Thurston--Veech construction}

We prove our main results using the Thurston--Veech construction. This construction of 
pseudo-Anosov maps appeared independently in two papers by Thurston and 
Veech~\cite{Th,Veech:construction}. 
\smallskip

\noindent
A \emph{multicurve} is a disjoint union of simple closed curves, 
and a pair of multicurves~$\alpha,\beta\subset S_g$ \emph{fills} the surface~$S_g$ 
if~$\alpha$ and~$\beta$ intersect transversally and if the 
complement~$S_g \setminus \left( \alpha \cup \beta \right)$ is a union 
of topological discs none of which is a bigon. This in particular implies 
that each pair~$\alpha_i$ and~$\beta_j$ of components realises the minimal 
number of intersection points within their respective isotopy classes.  
\smallskip

\noindent
Given a pair of filling 
multicurves~$\alpha, \beta \subset S_g$, the Thurston--Veech 
construction provides pseudo-Anosov mapping classes in the subgroup~$\langle T_\alpha, T_\beta \rangle$ 
of~$\mathrm{Mod}(S_g)$ generated by multitwists along the multicurves~$\alpha$ and~$\beta$. 
In his seminal 1988 Bulletin paper~\cite{Th}, Thurston provides the upper bound 
of~$6g-6$ on the algebraic degree of a pseudo-Anosov stretch factor~$\lambda(f)$ and 
claims, without proof, that {``\em the examples of~\cite[Theorem~7]{Th} show that this 
bound is sharp"}. The referenced examples are exactly the pseudo-Anosov maps 
in~$\langle T_\alpha, T_\beta\rangle$. 
\smallskip

\noindent
Margalit remarked in 2011 what Strenner wrote down in his article on stretch factor 
degrees~\cite{Strenner:algebraic}, namely that no proof of Thurston's claim has ever been 
published. We are finally able to substantiate Thurston's claim. Furthermore, we can even do so 
for pseudo-Anosov maps in the Torelli group. Our precise statement for the 
Thurston--Veech construction is the following.
 
\begin{thm}
\label{thm:main:1}
Let~$g\ge2$ and~$1\le d\le 3g-3$ be integers. Then there exists a pseudo-Anosov map on~$S_g$
arising from the Thurston--Veech construction with trace field degree~$d$ 
and stretch factor degree~$2d$. For~$g\ge 3$, the pseudo-Anosov maps can be 
chosen in the Torelli group~$\mathcal{I}(S_g)$. 
\end{thm}

\noindent
Clearly,~\Cref{thm:main:1} implies~\Cref{thm:trace} and~\Cref{thm:stretch} for~$g\ge 3$. 
Our proof of the Torelli case of~\Cref{thm:main:1} does not work for~$g=2$, and for this situation 
we directly prove~\Cref{thm:trace} and~\Cref{thm:stretch} by using ad hoc examples (see~\Cref{sec:proof:g2}).

\subsection{Proof strategy for~\Cref{thm:main:1}}
\label{sec:proof:strategy}
For a pair~$\alpha,\beta\subset S_g$ of filling multicurves, let 
$X = \left( |\alpha_i \cap \beta_j | \right)_{ij}$ be the matrix encoding the number 
of {geometric} intersections of the components of~$\alpha$ and~$\beta$. 
\smallskip

\noindent
The matrix~$XX^\top$ is primitive, hence by Perron--Frobenius theory its spectral 
radius equals its largest eigenvalue and is therefore an algebraic integer. 
Let~$d$ be its algebraic degree. We call the number~$d$ the {\em multicurve intersection 
degree} of~$\alpha$ and~$\beta$.
\smallskip

\noindent
Our proof is  based on the following existence result.

\begin{thm}
\label{thm:criterion:nonsplitting}
Let~$\alpha,\beta \subset S$ be a pair of filling multicurves having multicurve intersection 
degree~$d$. For~$\varepsilon \in \mathbb{Z}\setminus\{0\}$, there exists~$n \in \mathbb Z_{>0}$ 
such that the mapping class~$T^{n}_\alpha \circ T^{n\varepsilon}_\beta$ is pseudo-Anosov 
with stretch factor~$\lambda$ of degree~$2d$.
\end{thm}

\noindent
Assuming~\Cref{thm:criterion:nonsplitting}, what remains to be done in order to prove the first part 
of~\Cref{thm:main:1} is to construct all multicurve intersection degrees~$1\le d \le 3g-3$ 
on~$S_g$ for~$g\ge2$. By the Thurston--Veech construction, the trace 
field degree of the resulting pseudo-Anosov maps equals exactly the multicurve 
intersection degree of the multicurves~$\alpha$ and~$\beta$ used in the construction~\cite{Th,Veech:construction}.
Hence, we are done by setting~$\varepsilon =1$ in~\Cref{thm:criterion:nonsplitting}. 
\smallskip

\noindent
In order to prove the Torelli part of~\Cref{thm:main:1}, we construct the multicurves~$\alpha$ 
and~$\beta$ realising the multicurve intersection degrees~$1\le d \le 3g-3$ in such a way 
that~$T_\alpha \circ T_\beta^{-1}$ is an element of~$\mathcal{I}(S_g)$. 
We will ensure this by choosing multicurves~$\alpha$ and~$\beta$ which consist of components 
that are separating or that come in bounding pairs, where for each bounding pair one 
of the curves is a component of~$\alpha$ and the other is a component of~$\beta$. 
We can then finish the proof of~\Cref{thm:main:1} by setting~$\varepsilon = -1$ in~\Cref{thm:criterion:nonsplitting}.
To see this, note that if~$T_\alpha \circ T_\beta^{-1}$ is an element of~$\mathcal{I}(S_g)$,
then so is~$T_\alpha^n \circ T_\beta^{-n}$.
\smallskip

\noindent
We note that our examples of multicurves~$\alpha$ and~$\beta$ that we use in order to 
construct examples in the Torelli group~$\mathcal{I}(S_g)$ cannot yield multicurve intersection degrees 
greater than one on the surface~$S_2$. Indeed, there exist no bounding pairs on~$S_2$ 
and a multicurve can have at most one separating component. 
\smallskip 

\noindent
{The constructed examples in~\cite{Strenner:algebraic} do not belong to the Torelli group. Indeed, 
the first components of the considered multicurves 
$A$ and $B$~\cite[Figure 6.1]{Strenner:algebraic}  are nonzero in homology and they intersect nontrivially.}
\smallskip

\noindent
To find the suitable multicurves~$\alpha$ and~$\beta$ realising 
all possible multicurve intersection degrees~$1\le d \le 3g-3$ is the main technical contribution 
in this article.

\subsection{Proof of~\Cref{thm:trace} and~\Cref{thm:stretch} for $g=2$}
\label{sec:proof:g2}
For the genus two surface we give ad hoc examples. 
For this purpose, we use the flipper software~\cite{flipper} and start with the genus 
two surface with one puncture~$S_{2,1}$, see the figure below. In flipper, 
mapping classes are defined via Dehn twists along the curves~$a,b,c,d,e,f$.
\medskip

\begin{minipage}{.6\textwidth}
Consider the separating curve~$\gamma$ depicted in blue. By the chain relation~$T_\gamma = (T_a\circ T_f)^6$.
We consider the following three conjugates of~$T_\gamma$:
\begin{enumerate}
\item $T_1 = (T_{f}\circ T_a\circ T_b) \circ T_\gamma \circ (T_{f}\circ T_a\circ T_b)^{-1}$,
\item $T_2 = (T_{c}\circ T_b) \circ T_\gamma \circ (T_{c}\circ T_b)^{-1}$,
\item $T_3 = (T_{a}\circ T_b) \circ T_\gamma \circ (T_{a}\circ T_b)^{-1}$.
\end{enumerate}
\end{minipage}
\hfill
\begin{minipage}{.35\textwidth}
\def\svgwidth{145pt}
\begingroup%
  \makeatletter%
  \providecommand\color[2][]{%
    \errmessage{(Inkscape) Color is used for the text in Inkscape, but the package 'color.sty' is not loaded}%
    \renewcommand\color[2][]{}%
  }%
  \providecommand\transparent[1]{%
    \errmessage{(Inkscape) Transparency is used (non-zero) for the text in Inkscape, but the package 'transparent.sty' is not loaded}%
    \renewcommand\transparent[1]{}%
  }%
  \providecommand\rotatebox[2]{#2}%
  \newcommand*\fsize{\dimexpr\f@size pt\relax}%
  \newcommand*\lineheight[1]{\fontsize{\fsize}{#1\fsize}\selectfont}%
  \ifx\svgwidth\undefined%
    \setlength{\unitlength}{355.93286666bp}%
    \ifx\svgscale\undefined%
      \relax%
    \else%
      \setlength{\unitlength}{\unitlength * \real{\svgscale}}%
    \fi%
  \else%
    \setlength{\unitlength}{\svgwidth}%
  \fi%
  \global\let\svgwidth\undefined%
  \global\let\svgscale\undefined%
  \makeatother%
  \begin{picture}(1,0.40567815)%
    \lineheight{1}%
    \setlength\tabcolsep{0pt}%
    \put(0,0){\includegraphics[width=\unitlength]{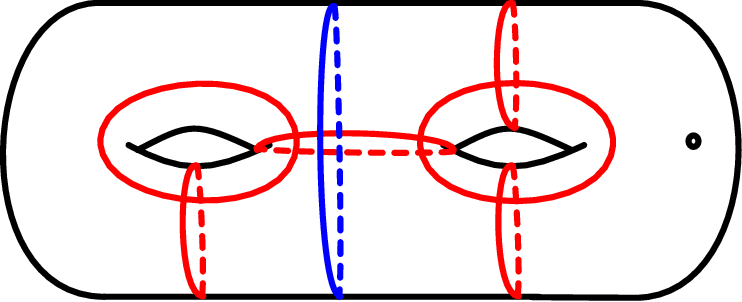}}%
    \put(0.07968608,0.20988332){\makebox(0,0)[lt]{\lineheight{1.25}\smash{\begin{tabular}[t]{l}$a$\end{tabular}}}}%
    \put(0.51143068,0.25329493){\makebox(0,0)[lt]{\lineheight{1.25}\smash{\begin{tabular}[t]{l}$b$\end{tabular}}}}%
    \put(0.8542235,0.21469972){\makebox(0,0)[lt]{\lineheight{1.25}\smash{\begin{tabular}[t]{l}$c$\end{tabular}}}}%
    \put(0.18349477,0.05049516){\makebox(0,0)[lt]{\lineheight{1.25}\smash{\begin{tabular}[t]{l}$f$\end{tabular}}}}%
    \put(0.72664694,0.32681021){\makebox(0,0)[lt]{\lineheight{1.25}\smash{\begin{tabular}[t]{l}$d$\end{tabular}}}}%
    \put(0.72381958,0.05381225){\makebox(0,0)[lt]{\lineheight{1.25}\smash{\begin{tabular}[t]{l}$e$\end{tabular}}}}%
    \put(0.48692287,0.05049516){\makebox(0,0)[lt]{\lineheight{1.25}\smash{\begin{tabular}[t]{l}$\gamma$\end{tabular}}}}%
  \end{picture}%
\endgroup%

\end{minipage}
 \medskip
 
\noindent Obviously $T_\gamma,T_1,T_2,T_3 \in \mathcal{I}(S_{2,1})$.
For each even degree $d \in \{2,4,6\}$, we exhibit a word in the above elements.  
We check that the mapping class is pseudo-Anosov with singularity pattern $(1,1,1,1;0)$, 
and we compute its stretch factor by using flipper~\cite{flipper}. The vector
$(1,1,1,1;0)$ means that the invariant foliations have four $3$-pronged type singularities, 
and a $2$-pronged one at the puncture. 
\medskip

\begin{center}
\begin{tabular}{p{5,5cm} | cp{2,5cm}}
pseudo--Anosov mapping class $[f]\in \mathrm{Mod}(S_{2,1})$ & minimal polynomial of $\lambda(f)$ \\
\hline
\hline
$T_\gamma\circ T_1 \circ T_2^{-1}$ & $t^2 - 66t + 1$ \\
$T_\gamma\circ T_1 \circ T_2$ &  $t^4 - 72t^3 + 110t^2 - 72t + 1$ \\
$T_\gamma\circ T_1 \circ T_2 \circ T_3$ & $t^6 - 266t^5 + 143t^4 - 204t^3 + 143t^2 - 266t + 1$
\end{tabular}
\end{center}
\medskip

\noindent In order to obtain elements in $\mathrm{Mod}(S_{2})$, we appeal to the forgetful map. 
Since our examples do not have a $1$-pronged singularity at the puncture, {we can fill it in order}
to get a pseudo-Anosov mapping class, with the same stretch factor, in 
$\mathrm{Mod}(S_{2})$, see~\cite[Lemma 2.6]{HironakaKin} for details. This completes the proof {of~\Cref{thm:stretch}} for~$g=2$. 
{Now for each example, the minimal polynomial of $\lambda(f)$ is reciprocal and has degree $2d$ (for $d=1,2,3$ respectively). Hence
the minimal polynomial of $\lambda(f) + \lambda(f)^{-1}$ has degree $d$ as required, and~\Cref{thm:trace} is proved.}

\subsection{Explicit examples}
\label{sec:explicit}
In~\cite[Problem 10.4]{Margalit:question}, Margalit asks for explicit examples 
of pseudo-Anosov maps with specific stretch factor degrees. Our construction of multicurves  
allows us to do so for small genera. 
\smallskip

\noindent
More precisely, we first provide an almost explicit construction of multicurves with intersection 
degrees~$1<d\le 3g-3$ for which the pseudo-Anosov mapping class~$T_\alpha \circ T_\beta$ has stretch 
factor degree $2d$. In the setting of~\Cref{thm:criterion:nonsplitting}, this means 
for~$\varepsilon=1$ one can choose~$n=1$. Our argument uses~\cite[Theorem~6]{LL} (see~\Cref{sec:proof} for details).
Unfortunately, for the multicurves we construct in the Torelli case, the criterion from \cite{LL} does not apply. Hence our 
need to use~\Cref{thm:criterion:nonsplitting} for the proof of~\Cref{thm:main:1}, which provides a slightly 
less explicit result. 
\smallskip

\noindent
Aided by the computer, we can subsequently find completely explicit examples of multicurves 
and therefore entirely explicit pseudo-Anosov mapping classes on~$S_g$ arising from the 
Thurston--Veech construction realising the maximal stretch factor degree~$6g-6$, for 
all genera~$g\le 201$. \smallskip

\noindent {We emphasise that this construction, contained in \Cref{sec:proof}, 
is independent of the proof of our main results, \Cref{thm:trace}, \Cref{thm:stretch} and~\Cref{thm:main:1}.}

\subsection{Odd degree stretch factors.}

\noindent
While~\Cref{thm:criterion:nonsplitting} provides the existence of field 
extensions~$\mathbb{Q}(\lambda) : \mathbb{Q}(\lambda+\lambda^{-1})$ 
of degree two for mapping classes in~$\langle T_\alpha, T_\beta \rangle$, 
realising extensions of degree one seems to be more mysterious. For example, 
Veech~\cite{Veech} discovered a family of Hecke groups~$\langle T_\alpha, T_\beta \rangle = 
\langle \left(\begin{smallmatrix} 1 & \lambda_q \\ 0 & 1 \end{smallmatrix}\right) , 
\left(\begin{smallmatrix} 1 & 0\\ -\lambda_q & 1 \end{smallmatrix}\right) \rangle$, 
where~$\lambda_q = 2\cos \pi/q$ for~$q\geq 3$. The genus of the surface~$S_g$ is~$(q-1)/2$ for 
odd~$q$. For~$q=7,9$ one can find stretch factors of degree one over the trace 
field~$\mathbb{Q}(\lambda_q)$: for instance $T_\alpha \circ T_\beta^{-1}$ is an example for $q=7$, 
and we refer to~\cite{Boulanger} for $q=9$. 
However, it is conjectured (see~\cite[Remark 9]{Hanson2008GeneralizedCF}) 
that stretch factors of degree one over~$\mathbb{Q}(\lambda_q)$ do not exist for odd~$q\geq 11$.
\smallskip

\noindent
It remains an open problem to construct odd stretch factor degrees in the 
Torelli group~$\mathcal{I}(S_g)$.

\subsection{On Thurston's upper bound}
One may impose restrictions on the foliations fixed by the pseudo-Anosov map, for example by prescribing 
the stratum, that is, the number of singularities as well as their orders as $k$-pronged singularities.
It turns out that Thurston's upper bound for the stretch factor degree becomes
\[
2g-2 + \# \{\textrm{odd singularities}\} \leq 2g-2 + 4g-4 = 6g-6.
\]
In the context of the Thurston--Veech construction, the number and type of singularities can be read off 
directly from the geometry of the complement~$S\setminus (\alpha \cup \beta)$: the number of~$k$-pronged 
singularities coincides with the number of~$2k$-gons in the complement. 
\smallskip

\noindent 
It is a consequence of our proof of~\Cref{thm:main:1} that the upper bound 
\[
2g-2 + \# \{\textrm{odd singularities}\}
\] 
for the stretch factor degree is sharp for every~$g\geq 3$, and that examples realising this upper bound 
can be taken in the Torelli group.

\subsection{A natural field extension from the perspective of curves}

A pair of filling multicurves~$\alpha,\beta \subset S_g$ naturally determines a bipartite graph whose 
vertices correspond to curve components and the number of edges between each pair of vertices 
equals the number of intersection points of the respective curve components. The adjacency 
matrix of this graph is~$\Omega= \left(\begin{smallmatrix}0 & X \\ X^\top & 0\end{smallmatrix}\right)$. 
Clearly, the square root~$\sqrt{\mu}$ of the spectral radius~$\mu$ of $XX^\top$ equals the spectral 
radius of~$\Omega$. We call the algebraic degree of~$\sqrt{\mu}$ the {\em multicurve bipartite degree} 
of~$\alpha$ and~$\beta$. Similarly to the field extension~$\mathbb{Q}(\lambda) : \mathbb{Q}(\lambda+\lambda^{-1})$, 
also the field extension~{$\mathbb{Q}(\sqrt{\mu}) : \mathbb{Q}(\mu)$}
has degree one or two. {From the point of view of the Thurston--Veech construction 
one has $\lambda(f) + \lambda(f)^{-1} = 2 + \mu$ for~$f=T_\alpha\circ T_\beta^{-1}$ (see the proof of~\Cref{thm:criterion:nonsplitting}).}
We prove the following result.

\begin{thm} 
\label{multicurvedegreethm:bipartite}
Every even integer~$2 \le 2d \le 6g-6$ is realised as a multicurve bipartite degree 
on~$S_g$ for~$g\ge2$.
\end{thm}

\subsection*{Organisation} In~\Cref{sec:new:nonsplitting} we prove~\Cref{thm:criterion:nonsplitting}, 
which is the nonsplitting criterion used to reduce~\Cref{thm:main:1} to the construction of certain 
kinds of pairs of multicurves. In~\Cref{sec:irreducibility:criterion} we introduce an 
irreducibility criterion for the characteristic polynomial of matrices of the form~$XX^\top$ 
which plays a central role throughout the rest of the article. Using this irreducibility criterion, 
we first give a proof of Thurston's claim as a warm-up in~\Cref{realisation_section}, before 
providing the multicurves needed to prove~\Cref{thm:main:1} in~\Cref{sec:non:orientable}. 
Finally, we provide some explicit examples in~\Cref{sec:proof}.

\subsection*{Acknowledgments} 
The authors thank Dan Margalit and Jean-Claude Picaud for inspiring discussions, Mark Bell for 
his help with the software flipper~\cite{flipper} and Curt McMullen for comments on an earlier version 
of the article. They are particularly grateful to Dan Margalit for pointing out a mistake in the formulation 
of the proof of~\Cref{thm:main:1} in an earlier version of the article.
{The authors also thank the anonymous referee for several useful suggestions, comments and corrections.}
The first author has been partially supported by the LabEx 
PERSYVAL-Lab (ANR-11-LABX-0025-01) funded by the French program Investissement 
d’avenir. The first author would like to thank the University of Fribourg and
the Center for Mathematical Modeling (CMM) at Universidad de Chile
for excellent working conditions.


\section{A nonsplitting criterion}
\label{sec:new:nonsplitting}
\noindent
In this section we prove~\Cref{thm:criterion:nonsplitting}, which is an algebraic criterion that allows us 
to deduce that the degree of the field extension~$\mathbb{Q}(\lambda(f)) : \mathbb{Q}(\lambda(f)+\lambda(f)^{-1})$ 
equals two for certain~$f$ which are a product of multitwists. Compare with~\cite[Theorem~6]{LL}. 
For convenience, 
we repeat the statement of~\Cref{thm:criterion:nonsplitting}:

\begin{nonumberthm}[\Cref{thm:criterion:nonsplitting}]
Let~$\alpha,\beta \subset S$ be a pair of filling multicurves having multicurve intersection degree~$d$. 
For every~$\varepsilon \in \mathbb{Z}\setminus\{0\}$, there exists~$n \in \mathbb Z_{>0}$ such that the mapping 
class~$T^{n}_\alpha \circ T^{n\varepsilon}_\beta$ is pseudo-Anosov with stretch factor~$\lambda$ 
of degree~$2d$.
\end{nonumberthm}

\noindent
{We need the following lemma which is a small extension of a result of Murty~\cite[Theorem 4]{Ram}. It is 
probably well-known but we could not  find a reference so we include a proof here.}

{\begin{lem}
\label{lm:almost:prime}
Let $Q(t) \in \mathbb Z[t]$ be a polynomial taking an integral square value at all but finitely many positive integer specialisations.
Then $Q(t)$ is the square of a polynomial in~$ \mathbb Z[t]$.
\end{lem}
\begin{proof}
Let~$N\in \mathbb N$ be a positive integer such that~$Q(n)$ is a square for every~$n\in \mathbb N$ with~$n>N$. 
Let~$n\in \mathbb Z$. Consider a prime number~$p$ such that~$n + p > N$. The integer~$Q(n+p)$ is a square by assumption, and  
its reduction modulo~$p$ is~$Q(n)$. Thus,~$Q(n) \in \mathbb Z$ is a quadratic residue for all but finitely many prime numbers. 
So the value~$Q(n)$ is a square in~$\mathbb Z$ (see~\cite[Theorem 1 (end of the proof)]{Marshall}). Thus we can apply~\cite[Theorem 4]{Ram} to conclude that $Q(t)$ is the square of a polynomial in~$\mathbb Z[t]$.
\end{proof}
}

\begin{proof}[Proof of~\Cref{thm:criterion:nonsplitting}]
By the Thurston--Veech construction, there exists a 
representation~$\rho:\left< T_\alpha, T_\beta \right>\to \mathrm{PSL}_2(\mathbb{R})$ mapping~$T_\alpha$ to the 
matrix~$\left(\begin{smallmatrix} 1 & r \\ 0 & 1 \end{smallmatrix}\right)$ and~$T_\beta$ to the 
matrix~$\left(\begin{smallmatrix} 1 & 0 \\ -r & 1 \end{smallmatrix}\right)$, where~$r^2 = \mu$ is 
the spectral radius of the matrix~$XX^\top$ for the multicurves~$\alpha$ and~$\beta$.
Furthermore, the stretch factor~$\lambda(f)$ of~$f \in \left< T_\alpha, T_\beta \right>$ equals the 
spectral radius of~$\rho(f)$. Now, let us consider the product of 
multitwists~$f=T^{2n}_\alpha \circ T^{2 n\varepsilon}_\beta$.
A direct computation provides that the trace of~$\rho(f)$ equals~$\mathrm{tr}(\rho(f)) = 2-\varepsilon(2nr)^2$. 
Thus,~$\lambda(f)+\lambda(f)^{-1} = |2-\varepsilon(2nr)^2|$ and 
hence~$\mathbb{Q}(\lambda(f) + \lambda(f)^{-1})=\mathbb Q(\mu)=K$. Note that by assumption, the 
degree of the field extension~$K : \mathbb{Q}$ is~$d$, the multicurve intersection 
degree of~$\alpha$ and~$\beta$.
\smallskip

\noindent
Since~$\lambda=\lambda(f)$ solves the quadratic equation~$t^2-(\lambda+\lambda^{-1})t+1=0$,~$\lambda$ 
has degree~$1$ or~$2$ over~$K$. All {that} we need to do is find~$n \in \mathbb Z_{>0}$
such that~$\lambda \not \in K$, or equivalently such that the 
discriminant~$D=(2-\varepsilon(2nr)^2)^2-4=16\cdot n^2\cdot ((n\varepsilon \mu)^2 - \varepsilon\mu)$ 
of the quadratic equation is not a square in~$K$. We proceed by contradiction. 
Let~$\mu'=\varepsilon\mu$ and let us assume that~$(n \mu')^2 - \mu'$ is a square in~$K=\mathbb Q(\mu')$ 
for every~$n > 0$. Since the expression {$(n \mu')^2 - \mu'$} is invariant under the transformation~$n \mapsto -n$, 
we can assume that {$(n \mu')^2 - \mu'$} is a square for every~$n \in \mathbb{Z}\setminus\{0\}$. 
\smallskip


\noindent
{Let~$q \in \mathbb{Q}[t]$ be the minimal polynomial of~$\frac{1}{\mu'}$ over~$\mathbb Q$. Our strategy is to use \Cref{lm:almost:prime} to show the following claim.} 
\smallskip

\noindent
{\emph{Claim. The polynomial~$q(t^2)\in \mathbb{Q}[t]$ has a double root in~$\mathbb C$. }}
\smallskip

\noindent
{This claim directly leads to a contradiction. Indeed, the polynomial~$q \in \mathbb{Q}[t]$ is irreducible, and thus only 
has simple roots. Now each root~$0\neq a\in \mathbb C$ of~$q$ gives rise to two 
distinct roots~$\pm \sqrt{a}$ of~$q(t^2)$, contradicting the claim. 
In order to finish the proof of \Cref{thm:criterion:nonsplitting}, it therefore remains to prove the claim.}
\smallskip

\noindent
\emph{Proof of Claim.}
Let~$P= a_dt^d+a_{d-1}t^{d-1}+\dots+a_1t+a_0 \in \mathbb Q[t]$ be the minimal polynomial 
of~$\mu'$ over~$\mathbb Q$ with $a_d=1$. The Thurston--Veech construction implies that~$\mu$ is an eigenvalue 
of a square matrix with integer coefficients, so~{$\mu'$ is also an eigenvalue of a square matrix with integer coefficients, 
implying that~$P \in \mathbb Z[t]$}. \smallskip

\noindent
Thus,~$\mu'$ and~$n^2\mu'-1$ are algebraic {integers}, {and we can consider their norms~$N(\cdot)$.} Obviously, for~$\mu'$ one has~$N(\mu')=(-1)^d a_0$. Let us compute the norm of $n^2\mu'-1$. {The monic  polynomial $n^{2d}P\left( \frac{t+1}{n^2} \right) \in \mathbb Z[t]$ has degree $d$ and admits $n^2\mu'-1$ as a root. Moreover the fields $\mathbb Q(\mu')$ and $\mathbb Q(n^2\mu'-1)$ coincide, so they have the same degree over $\mathbb Q$, namely $d$. Hence $n^{2d}P\left( \frac{t+1}{n^2} \right)$ is the minimal polynomial of~$n^2\mu'-1$}. Inspecting its constant term, we deduce
\[
N(n^2\mu'-1) = (-1)^d \sum_{k=0}^d a_kn^{2d-2k}.
\]
{Since the norm is multiplicative, we have
\[
N((n\mu')^2-\mu') = N(n^2\mu'-1)\cdot N(\mu') = a_0 \sum_{k=0}^d a_kn^{2d-2k} = Q(n^2),
\]
where}
\[
Q(t) = a_0 \sum_{k=0}^d a_k t^{d-k}.
\]
{Since by assumption $(n \mu')^2 - \mu'$ is a square for every~$n \in \mathbb{Z}\setminus\{0\}$, we deduce that~$Q(n^2)$ is a square for every~$n \in \mathbb{Z}\setminus\{0\}$. By~\Cref{lm:almost:prime}, $Q(t^2)$ is itself a square in $\mathbb Z[t]$.} \medskip

%
%

\noindent
To finish the proof, observe that~$Q(t) = a_0\cdot t^dP\left( \frac1{t}\right) \in {\mathbb Z}[t]$. 
In particular,~$Q\left( \frac1{\mu'}\right)=0$ and $Q$ has degree $d$ over $\mathbb Q$. 
Since~$\mu'$ and~$\frac1{\mu'}$ generate the same field extension~$K:\mathbb Q$, the minimal polynomial~$q$ of~$\frac{1}{\mu'}$ over~$\mathbb Q$ is also 
of degree~$d$. This implies that~$Q$ must be a rational multiple of~$q$.  
In particular,~$q(t^2)$ is a rational multiple of~$Q(t^2)$ and therefore has a double root in~$\mathbb C$.  
\end{proof}

\section{An irreducibility criterion}
\label{sec:irreducibility:criterion}

\noindent
The goal of this section is to present an algebraic criterion that allows us to deduce that  
certain characteristic polynomials of matrices of the form~$XX^\top$ are irreducible.
{We write~$E_{11}$ for the matrix with the top-left coefficient equal to~$1$ and all other coefficients zero. 
Furthermore, for a square matrix~$A$ with coefficients in a ring~$R$, we denote 
its characteristic polynomial by~$\chi_A = \det ( t\cdot \mathrm{Id} - A) \in R[t]$.}

\begin{prop}
\label{criterion_prop}
Let~$M$ be a square integer matrix, and let~$N$ 
be the principal submatrix of~$M$ obtained by deleting the first row and the first column. 
If~$M$ and~$N$ have no common eigenvalue, and if~$M$ has a  
simple eigenvalue~$\rho$, then the characteristic polynomial of~$\widetilde M = M + {c}y^pE_{11}$ 
is an irreducible element of~$\mathbb{Z}[t,y]$, for all~{$0\ne p \in \mathbb{N}$} and~$0\ne {c} \in\mathbb{Z}$.
\end{prop}

\begin{proof}
Our goal is to use Eisenstein's criterion on~$\chi_{\widetilde M} \in \mathbb{Z}[t,y] \cong \left(\mathbb{Z}[t]\right)[y]$, 
viewing it as a polynomial in the variable~$y$ and coefficients in~$\mathbb{Z}[t]$. {By multilinearity of the determinant,} 
we calculate $$\chi_{\widetilde M} (t,y)=\det(t\cdot \mathrm{Id}-{\widetilde M}) = -y^p {c}\chi_N(t) + \chi_M(t)$$
and notice that~${c}\chi_N$ and~$\chi_M$ are relatively prime in~$\mathbb{Z}[t]$. 
Indeed,~$\chi_M$ has leading coefficient~$+1$ and no root in common with~$\chi_N$ by our assumption 
that~$M$ and~$N$ have no eigenvalue in common. In particular, they have no common factor, which shows 
that~$\chi_{\widetilde M} \in  \left(\mathbb{Z}[t]\right)[y]$ is primitive. In order to apply Eisenstein's criterion, 
let~$\mu_\rho \in \mathbb{Z}[t]$ be the minimal polynomial of the simple eigenvalue~$\rho$ of~$M$. 
By assumption,~$\mu_\rho$ divides~$\chi_M$ exactly once, but it does not divide~$\chi_N$ since~$\chi_M$ 
and~$\chi_N$ have no common root. In particular, Eisenstein's criterion applies to show that 
the polynomial~$\chi_{\widetilde M} \in  \left(\mathbb{Z}[t]\right)[y] \cong \mathbb{Z}[t,y]$ is irreducible. 
\end{proof}

\begin{remark}\label{rk:root} 
\emph{
In the previous statement, one can easily replace $\chi_{\widetilde M}(t)$ by $\chi_{\widetilde M}(t^n)$ for any integer $n>0$. Indeed
$\chi_M(t^n)$ and $\chi_N(t^n)$ are still coprime and $\mu_\rho(t^n)$ divides~$\chi_M(t^n)$ exactly once, so 
Eisenstein's criterion applies.}
\end{remark}

\begin{remark}\emph{
Oscillatory matrices satisfy a stronger version of Perron--Frobenius theory, namely \emph{all} the 
eigenvalues are positive real, simple, and they strictly interlace when taking a principal submatrix~\cite{Ando}. 
Hence,~\Cref{criterion_prop} applies very cleanly to this class of matrices.}
\end{remark}

\noindent
We use~\Cref{criterion_prop} on the following two cases (\Cref{add_one_genus_lemma} and~\Cref{double_genus_lemma}).
\begin{lem}
\label{add_one_genus_lemma}
For $n\geq 1$, let
\[
N=
\left(
\begin{array}{c|ccc}
a_1 & a_2 & \dots & a_n \\
\hline 
a_2 &  &  &  \\
\vdots  &  & \ast &   \\
a_n &  &  & 
\end{array}
\right),
\hspace{0.5cm}
M=
\left(
\begin{array}{c|ccc}
0 & \alpha a_1 & \dots &\alpha a_n \\
\hline 
\alpha a_1 &  &  &  \\
\vdots  &  & N &   \\
\alpha a_n &  &  & 
\end{array}
\right)
\]
be square integer matrices with~$a_1\ge1$ and {$0 \ne \alpha \in \mathbb{Q}$}. 
If~$M$ is nonnegative and irreducible, and if~$\chi_N \in \mathbb{Z}[t]$ is irreducible, 
then the characteristic polynomial of~$\widetilde M = M + {c}y^2E_{11}$ is irreducible in~$\mathbb{Z}[t,y]$  
for all~$0\ne {c} \in\mathbb{Z}$.
\end{lem} 

\begin{proof}
In order to use~\Cref{criterion_prop}, we need to show that~$M$ has a simple eigenvalue 
and that~$M$ and~$N$ share no eigenvalue. The former holds since~$M$ is nonnegative and irreducible, 
and in particular has a Perron--Frobenius eigenvalue which is simple, {see for example~\cite{Gantmacher}}. 
For the latter, we compute
$$\chi_{ M} (t) = t\chi_N(t) + q(t),$$ 
where~$q(t) \in \mathbb{Z}[t]$ is of degree at most~$n-1$. We claim that it is not the zero polynomial either. 
Indeed, we directly verify 
\begin{align*}
q(0) &= \det 
\left(
\begin{array}{c|ccc}
0 & - \alpha a_1 & \dots & -\alpha a_n \\
\hline 
-\alpha a_1 &  &  &  \\
\vdots  &  & - N &   \\
-\alpha a_n &  &  & 
\end{array}
\right)  \\
& =
\det 
\left(
\begin{array}{c|ccc}
\alpha^2 a_1 & 0 & \dots & 0 \\
\hline 
0 &  &  &  \\
\vdots  &  & - N &   \\
0 &  &  & 
\end{array}
\right) 
= \pm \alpha^2 a_1\det N \neq 0.
 \end{align*}
 Now if there existed a common root~$\lambda\in \mathbb{C}$ of~$\chi_M$ and~$\chi_N$,
 then~$\lambda$ would also be a root of~$q(t)$, 
 {as the equation~$\chi_{ M} (t) = t\chi_N(t) + q(t)$ implies~$q(\lambda) = \chi_M(\lambda) -\lambda\chi_N(\lambda) = 0$.
 Since~$\chi_N$ is irreducible, it is the minimal polynomial of~$\lambda$ over~$\mathbb{Q}$, and we assume it to be of degree~$n$. 
 But~$q(t) \in \mathbb{Z}[t] $ is a nonzero polynomial of degree at most~$n-1$ with root~$\lambda$, which is impossible 
 by the definition of the minimal polynomial. Hence,~$\chi_M$ and~$\chi_N$ have no common root~$\lambda\in \mathbb{C}$.}
\end{proof}
 
\begin{lem}
\label{double_genus_lemma}
For $n,m\geq 1$, let
\[
A=
\left(
\begin{array}{c|ccc}
a_1 & a_2 & \dots & a_n \\
\hline 
a_2 &  &  &  \\
\vdots  &  & \ast &   \\
a_n &  &  & 
\end{array}
\right),
\hspace{0.5cm}
B=
\left(
\begin{array}{c|ccc}
b_1 & b_2 & \dots & b_m \\
\hline 
b_2 &  &  &  \\
\vdots  &  & \ast &   \\
b_m &  &  & 
\end{array}
\right)
\]
be square integer matrices of size~$n$ and~$m$, respectively, with~$a_1, b_1 \ge 1$. 
Furthermore, let~$\alpha,\beta \ne 0$ such that 
\[
M=
\left(
\begin{array}{c|ccc|ccc}
0 & \alpha a_1 & \dots & \alpha a_n & \beta b_1 & \dots & \beta b_m \\
\hline 
\alpha a_1 &  &  &  & & &\\
\vdots  &  & A & & & &  \\
\alpha a_n &  &  & & & & \\
\hline
\beta b_1 & & & & & & \\
\vdots & & & & & B & \\
\beta b_m & & & & & &
\end{array}
\right)
\]
is a matrix with integer coefficients. If~$M$ is nonnegative and irreducible, and if~$\chi_A, \chi_B \in \mathbb{Z}[t]$ are irreducible 
and distinct, then the characteristic polynomial of~$\widetilde M = M + {c}y^2E_{11}$ is irreducible 
in~$\mathbb{Z}[t,y]$ for all~$0\ne {c} \in \mathbb{Z}$. 
\end{lem}

\begin{proof}
The proof is similar to the proof of~\Cref{add_one_genus_lemma}: the only thing to verify 
is that no eigenvalue of~$A$ or of~$B$ is also an eigenvalue of~$M$. Again, we compute
$$\chi_{ M} (t) = t\chi_A(t)\chi_B(t) \pm q_1(t)\chi_B(t) \pm q_2(t)\chi_A(t),$$ 
where~$q_1(t) \in \mathbb{Z}[t]$ is of degree at most~$n-1$ and~$q_2(t) \in \mathbb{Z}[t]$ is of 
degree at most~$m-1$. This is seen by developing the first column of the matrix~$tI-M$. The first 
coefficient is responsible for the summand~$t\chi_A(t)\chi_B(t)$, the next~$n$ coefficients are 
responsible for the summand~$\pm q_1(t)\chi_B(t)$ and the final~$m$ coefficients are responsible 
for the summand~$\pm q_2(t)\chi_A(t)$. We claim that neither among~$q_1(t)$ and~$q_2(t)$ is 
the zero polynomial. Indeed, by developing the first column 
of the matrix $t\mathrm{I}-M$, and evaluating at $t=0$, we get
\begin{align*}
q_1(0) &= \det 
\left(
\begin{array}{c|ccc}
0 & -\alpha a_1 & \dots & -\alpha a_n \\
\hline 
-\alpha a_1 &  &  &  \\
\vdots  &  & - A &   \\
-\alpha a_n &  &  & 
\end{array}
\right) \\
& =
\det 
\left(
\begin{array}{c|ccc}
\alpha^2 a_1 & 0 & \dots & 0 \\
\hline 
0 &  &  &  \\
\vdots  &  & - A &   \\
0 &  &  & 
\end{array}
\right) 
= \pm \alpha^2 a_1\det A,  
 \end{align*}
 which is not zero since $\chi_A$ is irreducible. Similarly,~$q_2(0) \ne 0$. 
 Now if there existed a common root~$\lambda\in \mathbb{C}$ of~$\chi_M$ and~$\chi_A$,
 then~$\lambda$ would also be a root of either~$\chi_B$ or~$q_1$. Since~$\chi_A$ and~$\chi_B$ 
 are irreducible and distinct, we must have~$q_1(\lambda) = 0$. But since~$\chi_A$ 
 is irreducible of degree~$n$, and~$q_1(t)$ is a nonzero polynomial of degree at most~$n-1$, 
 this is impossible. Similarly, no root of~$\chi_B$ can be a root of~$\chi_M$, which concludes 
 the proof.
\end{proof}

\begin{remark}
\label{k_blocks_remark}
\emph{One could formulate \Cref{double_genus_lemma} with~$k\ge2$ blocks $A_1,\dots,A_k$
of respective sizes~$n_1,\dots,n_k$, instead of~$k=2$. In this case, all the~$k$ characteristic polynomials 
$\chi_{A_i}$ should be irreducible and pairwise distinct. The argument is identical by considering
$$
\chi_{ M} (t) = t \prod_{i=1}^k \chi_{A_i} + \sum_{i=1}^k \pm q_i(t) \prod_{j\neq i} \chi_{A_j},
$$ 
where~$q_i(t) \in \mathbb{Z}[t]$ is of degree at most~$n_i-1$ and nonzero.}
\end{remark}

\section{Warm-up: proof of Thurston's claim}
\label{realisation_section}

\noindent
As a first illustration of our method, it is the goal of this section to provide
a pair of filling multicurves~$\alpha$ and~$\beta$ on~$S_g$ with multicurve intersection degree~$3g-3$. 
By~\Cref{thm:criterion:nonsplitting}, this validates Thurston's claim that the product of two 
multitwists can realise the maximal possible algebraic degrees of stretch factors:~$6g-6$. 
\smallskip

\noindent
Recall that the matrix encoding the number of intersections of the components of~$\alpha$ 
and~$\beta$ is denoted by~$X = \left( |\alpha_i \cap \beta_j | \right)_{ij}$ (see~\Cref{sec:proof:strategy}). 
In order to read off the matrix~$XX^\top$ from our figures, we use the following formula for 
its coefficients, which is a direct consequence of the definition of matrix multiplication: 
\[
(XX^\top)_{ij} = \sum_k |\alpha_i \cap \beta_k | \cdot | \beta_k \cap \alpha_j |.
\] 

\noindent
We start by realising, on the surface of genus~$g\ge 1$ with~$2$ boundary components, 
a pair of filling multicurves~$\alpha$ and~$\beta$ such that~$\chi_{XX^\top} \in \mathbb{Z}[t]$ 
is irreducible and of degree~$3g-1$. {On a surface with boundary, we say that a pair of multicurves 
is filling if they intersect transversally and if their complement consists of boundary-parallel 
annuli and of discs none of which is a bigon.}
We proceed by induction on~$g$.

\subsubsection*{For~$g=1$ with two boundary components} 
\label{g1_section}
We consider the {filling pair of} multicurves~$\alpha$ and~$\beta$ shown in~\Cref{genusone}, 
where one of the components of~$\beta$ has~$y-1$ parallel copies. Here, we think of~$y$ 
as a variable that we specify {to a value in~$\mathbb{Z}_{\ge 2}$} later on. 
\begin{figure}[h]
\def\svgwidth{120pt}
\begingroup%
  \makeatletter%
  \providecommand\color[2][]{%
    \errmessage{(Inkscape) Color is used for the text in Inkscape, but the package 'color.sty' is not loaded}%
    \renewcommand\color[2][]{}%
  }%
  \providecommand\transparent[1]{%
    \errmessage{(Inkscape) Transparency is used (non-zero) for the text in Inkscape, but the package 'transparent.sty' is not loaded}%
    \renewcommand\transparent[1]{}%
  }%
  \providecommand\rotatebox[2]{#2}%
  \newcommand*\fsize{\dimexpr\f@size pt\relax}%
  \newcommand*\lineheight[1]{\fontsize{\fsize}{#1\fsize}\selectfont}%
  \ifx\svgwidth\undefined%
    \setlength{\unitlength}{256.68221301bp}%
    \ifx\svgscale\undefined%
      \relax%
    \else%
      \setlength{\unitlength}{\unitlength * \real{\svgscale}}%
    \fi%
  \else%
    \setlength{\unitlength}{\svgwidth}%
  \fi%
  \global\let\svgwidth\undefined%
  \global\let\svgscale\undefined%
  \makeatother%
  \begin{picture}(1,1.24007265)%
    \lineheight{1}%
    \setlength\tabcolsep{0pt}%
    \put(0,0){\includegraphics[width=\unitlength]{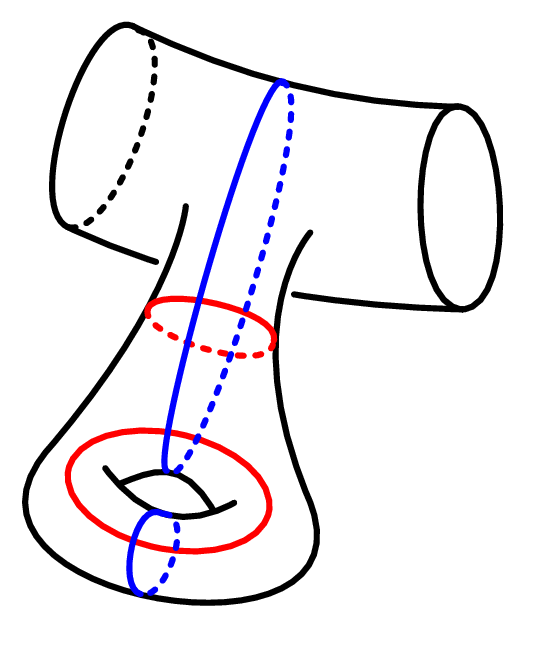}}%
    \put(0.20303637,0.01249419){\makebox(0,0)[lt]{\lineheight{1.25}\smash{\begin{tabular}[t]{l}$y-1$ copies\end{tabular}}}}%
    \put(0.03412539,0.47581523){\makebox(0,0)[lt]{\lineheight{1.25}\smash{\begin{tabular}[t]{l}$\alpha_2$\end{tabular}}}}%
    \put(0.56260138,0.57846982){\makebox(0,0)[lt]{\lineheight{1.25}\smash{\begin{tabular}[t]{l}$\alpha_1$\end{tabular}}}}%
  \end{picture}%
\endgroup%

\caption{Two multicurves~$\alpha$ (in red) and~$\beta$ (in blue) on the surface of genus 
one with two boundary components. The multicurve~$\beta$ contains~$y-1$ parallel 
copies of one of its components.}
\label{genusone}
\end{figure}
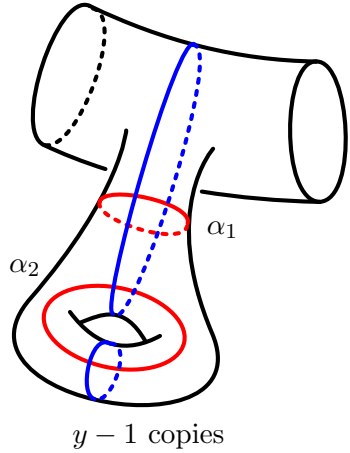

\noindent
{Observe that $X$ is a matrix of size $2\times y$ (the multicurve $\beta$ has $y$ components).} 
We directly calculate 
\[ 
XX^\top = 
\begin{pmatrix}
4 & 2 \\
2 & y
\end{pmatrix}.
\]
We have~$\chi_{XX^\top}(t) = t^2 -(4+y)t+4(y-1)$ with discriminant~$y^2-8y+32$, which is not a square if~$y\ge12$. 
Indeed, in this case we have
\[
(y-3)^2 = y^2 - 6y + 9 > y^2 - 8y + 32 > y^2 -8y + 16 = (y-4)^2.
\]
In particular, for~$y\ge 12$ the polynomial~$\chi_{XX^\top}$ is irreducible.

\subsubsection*{For~$g> 1$ and two boundary components} 
\label{inductive_step}
For the inductive step, assume we have constructed 
on the surface of genus~$g\ge 1$ with~$2$ boundary components a {filling} pair of multicurves~$\alpha', \beta'$ such 
that the characteristic polynomial~$\chi' = \chi_{XX^\top} \in \mathbb{Z}[t]$ is irreducible and of degree~$3g-1$. 
Furthermore, we inductively assume that~$\alpha'_1$ is a simple closed curve that encircles all the handles 
of the surface, as illustrated in~\Cref{power2induction}. Take a surface of genus~$1$ and two 
boundary components, as in the case of genus~$g=1$, see~\Cref{genusone}, 
and denote its {filling pair of} multicurves by~$\alpha''$ and~$\beta''$. {We assume that the component of~$\beta''$ 
with multiple parallel copies has~$b-1$ of them.} 
Now glue its right boundary component to the left boundary component of the surface of genus~$g$, 
and add two new curves~$\alpha_0$ and~$\beta_0$ to the multicurves. 
The curve~$\alpha_0$ encircles all the handles of the newly formed surface, 
and the curve~$\beta_0$ twice intersects~$\alpha_0$ but no 
other multicurve component. Again, see~\Cref{power2induction} for an illustration. 
{By construction, the pair of multicurves~$\alpha, \beta$ is again filling.}

\begin{figure}[h]
\def\svgwidth{330pt}
\begingroup%
  \makeatletter%
  \providecommand\color[2][]{%
    \errmessage{(Inkscape) Color is used for the text in Inkscape, but the package 'color.sty' is not loaded}%
    \renewcommand\color[2][]{}%
  }%
  \providecommand\transparent[1]{%
    \errmessage{(Inkscape) Transparency is used (non-zero) for the text in Inkscape, but the package 'transparent.sty' is not loaded}%
    \renewcommand\transparent[1]{}%
  }%
  \providecommand\rotatebox[2]{#2}%
  \newcommand*\fsize{\dimexpr\f@size pt\relax}%
  \newcommand*\lineheight[1]{\fontsize{\fsize}{#1\fsize}\selectfont}%
  \ifx\svgwidth\undefined%
    \setlength{\unitlength}{476.51717061bp}%
    \ifx\svgscale\undefined%
      \relax%
    \else%
      \setlength{\unitlength}{\unitlength * \real{\svgscale}}%
    \fi%
  \else%
    \setlength{\unitlength}{\svgwidth}%
  \fi%
  \global\let\svgwidth\undefined%
  \global\let\svgscale\undefined%
  \makeatother%
  \begin{picture}(1,0.32085707)%
    \lineheight{1}%
    \setlength\tabcolsep{0pt}%
    \put(0,0){\includegraphics[width=\unitlength]{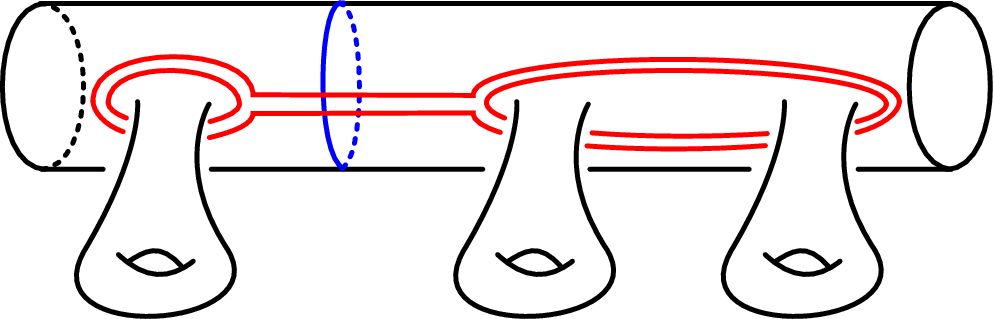}}%
    \put(0.62933222,0.20655681){\makebox(0,0)[lt]{\lineheight{1.25}\smash{\begin{tabular}[t]{l}$\cdots$\end{tabular}}}}%
    \put(0.2819539,0.27863436){\makebox(0,0)[lt]{\lineheight{1.25}\smash{\begin{tabular}[t]{l}$\beta_0$\end{tabular}}}}%
    \put(0.42531897,0.24515474){\makebox(0,0)[lt]{\lineheight{1.25}\smash{\begin{tabular}[t]{l}$\alpha_0$\end{tabular}}}}%
    \put(0.71914404,0.21049776){\makebox(0,0)[lt]{\lineheight{1.25}\smash{\begin{tabular}[t]{l}$\alpha_1'$\end{tabular}}}}%
    \put(0.15471844,0.21207782){\makebox(0,0)[lt]{\lineheight{1.25}\smash{\begin{tabular}[t]{l}$\alpha_1''$\end{tabular}}}}%
  \end{picture}%
\endgroup%

\caption{Two surfaces of {genus~$1$ and~$g$, respectively}, and two boundary components, glued together along one of their boundary components.
The curves~$\alpha'_1$ and~$\alpha''_1$ are shown, each encircling all the handles of their respective surface. The new curve~$\alpha_0$ 
encircles all the handles of the newly formed surface, and the new curve~$\beta_0$ runs along the glued boundary component. }
\label{power2induction}
\end{figure}
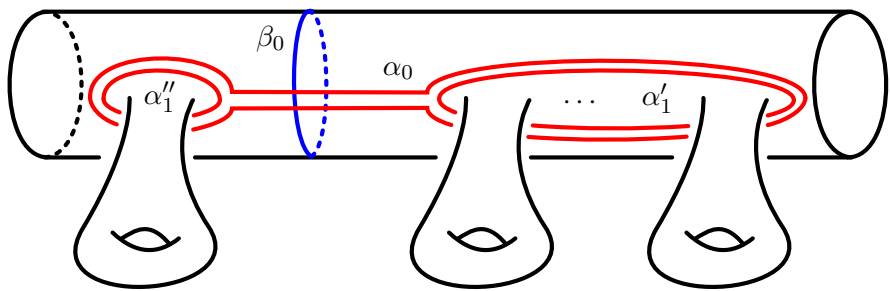

\noindent
Let~$A$ be the matrix~$XX^\top$ for the pair of multicurves~$\alpha', \beta'$, and let~$B$ be the matrix~$XX^\top$ for the 
pair of multicurves~$\alpha'', \beta''$. We define the multicurves~
\begin{align*}
\alpha &= \alpha_0 \cup \alpha' \cup \alpha'', \\
\beta &= \beta_0 \cup \beta' \cup \beta''.
\end{align*}
A quick computation gives
\[
A=
\left(
\begin{array}{c|ccc}
a_1 & a_2 & \dots & a_n \\
\hline 
a_2 &  &  &  \\
\vdots  &  & \ast &   \\
a_n &  &  & 
\end{array}
\right),
\hspace{0.5cm}
B=
\left(
\begin{array}{cccc}
4 & 2   \\
2 & b
\end{array}
\right).
\]
Let us choose~$b$ such that~$\chi_B$ is irreducible and distinct from~$\chi_A$. 
{We note~$a_1 = 4a$ for a positive integer~$a$. Indeed, since~$\alpha'_1$ is separating, 
its geometric intersection number with each component~$\beta'_k$ is a multiple of two, say~$2x_k$. In particular, 
we have $$a_1 = \sum_k |\alpha'_1 \cap \beta'_k | ^2 = 4 \sum_k x_k = 4a.$$} In the multicurve~$\beta$, 
we take~$y^2 - a - 1\ge 1$ parallel copies of~$\beta_0$, for~$y>0$ large enough. 
The matrix~$XX^\top$ for the multicurves~$\alpha$ and~$\beta$ takes the form
\[
XX^\top=
\left(
\begin{array}{c|ccc|cc}
4y^2 & a_1 & \dots & a_n & 4 & 2  \\
\hline 
a_1 &  &  &  & & \\
\vdots  &  & A & & &   \\
a_n &  &  & & & \\
\hline
4 & & & & 4 & 2  \\
2 & & & & 2 & b  
\end{array}
\right).
\]
{It is a crucial feature of the Thurston--Veech construction that for a filling pair of multicurves, 
such a matrix~$XX^\top$ is always primitive and hence irreducible~\cite{Th,Veech:construction}.} 
By~\Cref{double_genus_lemma}, the polynomial~$\chi_{XX^\top}\in\mathbb{Z}[t,y]$ is irreducible (recall that $\chi_{A}$ and~$\chi_{B}$ are irreducible). Hence, 
by Hilbert's irreducibility theorem, there exist infinitely many specifications of~$y$ (and in particular 
infinitely many specifications of~$y$ such that~$y^2-a-1>0$) 
with~$\chi_{XX^\top}\in\mathbb{Z}[t]$ irreducible.
{The degree of this polynomial equals the size of the matrix~$XX^\top$, which equals} 
$$3g-1 + 3 = 3(g+1)-1,$$ 
which is exactly what we wanted to show.
Finally, to justify our inductive assumption {on the curve~$\alpha'_1$, 
note that~$\alpha_0$ is again a curve that encircles all handles of the newly built surface of genus~$g+1$ with two boundary components.}

\subsubsection*{The closed case for~$g\ge2$.}
\label{closedcase}
Take any example of a {filling} pair of multicurves~$\alpha'$ and~$\beta'$ we constructed on the 
surface of genus~$g-1\ge1$ with two boundary components. {In particular, the curve~$\alpha'_1$ encircles 
all the~$g-1$ handles of the surface, and if  
\[
A=
\left(
\begin{array}{c|ccc}
a_1 & a_2 & \dots & a_n \\
\hline 
a_2 &  &  &  \\
\vdots  &  & \ast &   \\
a_n &  &  & 
\end{array}
\right)
\]
is the matrix~$XX^\top$ for the multicurves~$\alpha'$ and~$\beta'$, 
we have~$a_1 = 4a$ and~$\chi_A (t)$ is irreducible.}
We identify the two boundary components of the surface to increase the genus by one. 
Let~$\alpha_0$ be a longitude of the created handle, and let~$\beta_0$ run along the 
glued boundary, {see~\Cref{closed}}. 
Define the two new multicurves
\begin{align*}
\alpha &= \alpha_0 \cup \alpha' \\
\beta &= \beta_0 \cup \beta',
\end{align*}
where we take~$y^2-a$ copies of~$\beta_0$.
\begin{figure}[h]
\def\svgwidth{330pt}
\begingroup%
  \makeatletter%
  \providecommand\color[2][]{%
    \errmessage{(Inkscape) Color is used for the text in Inkscape, but the package 'color.sty' is not loaded}%
    \renewcommand\color[2][]{}%
  }%
  \providecommand\transparent[1]{%
    \errmessage{(Inkscape) Transparency is used (non-zero) for the text in Inkscape, but the package 'transparent.sty' is not loaded}%
    \renewcommand\transparent[1]{}%
  }%
  \providecommand\rotatebox[2]{#2}%
  \newcommand*\fsize{\dimexpr\f@size pt\relax}%
  \newcommand*\lineheight[1]{\fontsize{\fsize}{#1\fsize}\selectfont}%
  \ifx\svgwidth\undefined%
    \setlength{\unitlength}{476.51717061bp}%
    \ifx\svgscale\undefined%
      \relax%
    \else%
      \setlength{\unitlength}{\unitlength * \real{\svgscale}}%
    \fi%
  \else%
    \setlength{\unitlength}{\svgwidth}%
  \fi%
  \global\let\svgwidth\undefined%
  \global\let\svgscale\undefined%
  \makeatother%
  \begin{picture}(1,0.32085707)%
    \lineheight{1}%
    \setlength\tabcolsep{0pt}%
    \put(0,0){\includegraphics[width=\unitlength]{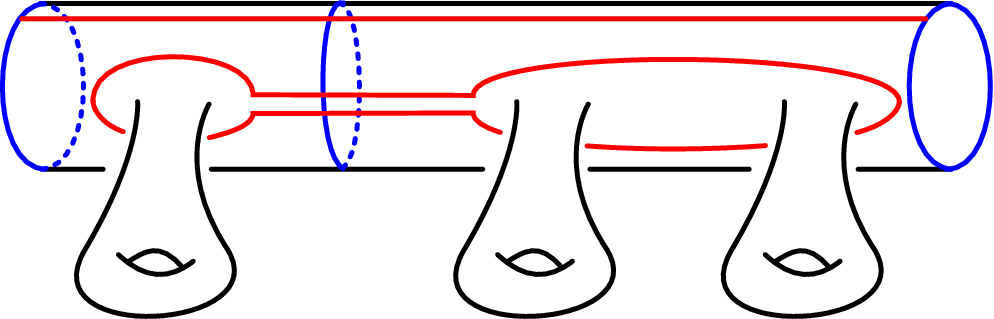}}%
    \put(0.65766279,0.21285249){\makebox(0,0)[lt]{\lineheight{1.25}\smash{\begin{tabular}[t]{l}$\cdots$\end{tabular}}}}%
    \put(0.27251038,0.25974732){\makebox(0,0)[lt]{\lineheight{1.25}\smash{\begin{tabular}[t]{l}$\beta'_1$\end{tabular}}}}%
    \put(0.41272761,0.17275439){\makebox(0,0)[lt]{\lineheight{1.25}\smash{\begin{tabular}[t]{l}$\alpha'_1$\end{tabular}}}}%
    \put(0.02380985,0.19863332){\makebox(0,0)[lt]{\lineheight{1.25}\smash{\begin{tabular}[t]{l}$\beta_0$\end{tabular}}}}%
    \put(0.93710655,0.19361453){\makebox(0,0)[lt]{\lineheight{1.25}\smash{\begin{tabular}[t]{l}$\beta_0$\end{tabular}}}}%
    \put(0.44651961,0.2751699){\makebox(0,0)[lt]{\lineheight{1.25}\smash{\begin{tabular}[t]{l}$\alpha_0$\end{tabular}}}}%
  \end{picture}%
\endgroup%

\caption{{The closed surface of genus~$g$ is obtained by identifying the left and the right boundary component of the 
depicted closed surface of genus~$g-1$ with two boundary components. The curve~$\alpha_0$ is a longitude of the created handle, 
and the curve~$\beta_0$ runs along the glued boundary.}}
\label{closed}
\end{figure}
{By construction, the pair of multicurves~$\alpha, \beta$ is again filling.} 
The 
matrix~$XX^\top$ for the multicurves~$\alpha$ and~$\beta$ takes the form
\[
XX^\top=
\left(
\begin{array}{c|ccc}
y^2 & \frac{a_1}{2} & \dots & \frac{a_n}{2} \\
\hline 
\frac{a_1}{2} &  &  &  \\
\vdots  &  & A &   \\
\frac{a_n}{2} &  &  & 
\end{array}
\right),
\]
 and~$\chi_{XX^\top}\in\mathbb{Z}[t,y]$ is irreducible by~\Cref{add_one_genus_lemma}.
By Hilbert's irreducibility theorem, there exist infinitely many specifications of~$y$ (and in particular 
infinitely many specifications of~$y$ such that~$y^2-a>0$) with~$\chi_{XX^\top}\in\mathbb{Z}[t]$ irreducible. 
{The degree of this polynomial equals the size of the matrix~$XX^\top$, which equals}~$$3(g-1)-1 + 1 = 3g-3.$$  
\smallskip

\section{Proof of~\Cref{thm:main:1}} 
\label{sec:non:orientable}

\noindent
The goal of this section is to realise every positive integer~$d \le 3g-3$ as the multicurve intersection 
degree of a pair of multicurves~$\alpha,\beta \subset S$ on~$S_g$ for~$g\ge 3$ in such a way that 
the multicurves~$\alpha$ and~$\beta$ consist of components 
that are separating or that come in bounding pairs, where for each bounding pair one 
of the curves is a component of~$\alpha$ and the other is a component of~$\beta$. 
{This guarantees that~$T_\alpha \circ T_\beta^{-1}$ is an element of the Torelli group~$\mathcal{I}(S_g)$.
Indeed, on the homology level, the action of a Dehn twist along a curve~$\gamma$ is given 
by~$[\delta] \mapsto [\delta] + i(\delta, \gamma)[\gamma]$, where~$i(\cdot,\cdot)$ is the algebraic 
intersection form on~${H}_1(S_g;\mathbb{Z})$. It follows that a Dehn twist along a separating curve 
is en element of~$\mathcal{I}(S_g)$, but also each composition of Dehn twists~$T_{\alpha_i} \circ T_{\beta_j}^{-1}$ 
for a bounding pair~$\alpha_i$ and~$\beta_j$. Since all the individual Dehn twists~$T_{\alpha_i}$ commute, 
as well as all the individual Dehn twists~$T_{\beta_j}$, it follows that if the multicurves~$\alpha$ and~$\beta$ consist of components 
that are separating or that come in bounding pairs, where for each bounding pair one 
of the curves is a component of~$\alpha$ and the other is a component of~$\beta$, 
the mapping class~$T_\alpha \circ T_\beta^{-1}$ is an element of~$\mathcal{I}(S_g)$,
and hence so is the mapping class~$T_\alpha^n \circ T_\beta^{-n}$.} 
\Cref{thm:criterion:nonsplitting} then provides~\Cref{thm:main:1} in the case~$g\ge 3$. 
For the case~$g=2$ we note that the statement is proved for~$d=3g-3 = 3$ in~\Cref{realisation_section},
and for~$d \le 2 = g$ it is proved in~\cite{LL}.
\smallskip

\noindent 
We start with the maximal degree~$3g-3$ and then discuss how to adapt the 
construction in order to realise smaller degrees.
\smallskip

\subsection{Multicurve intersection degree~$3g-3$}
\label{maximal_Torelli}
We start by realising, on the surface of genus~$g\ge 2$ with one boundary component, 
a filling pair of multicurves~$\alpha$ and~$\beta$ such that~$\chi_{XX^\top} \in \mathbb{Z}[t]$ 
is irreducible and of degree~$3g-2$, in such a way that 
the multicurves~$\alpha$ and~$\beta$ consist of components 
that are separating or that come in bounding pairs, where for each bounding pair one 
of the curves is a component of~$\alpha$ and the other is a component of~$\beta$. 
The construction is done by induction on the genus~$g\geq 2$.

\subsubsection{For~$g=2$ with one boundary component.}
We consider the two multicurves~$\alpha$ and~$\beta$ shown in 
\Cref{genustwo}, {with~$y$ parallel copies of the component~$\beta_1$}.
\begin{figure}[h]
\def\svgwidth{280pt}
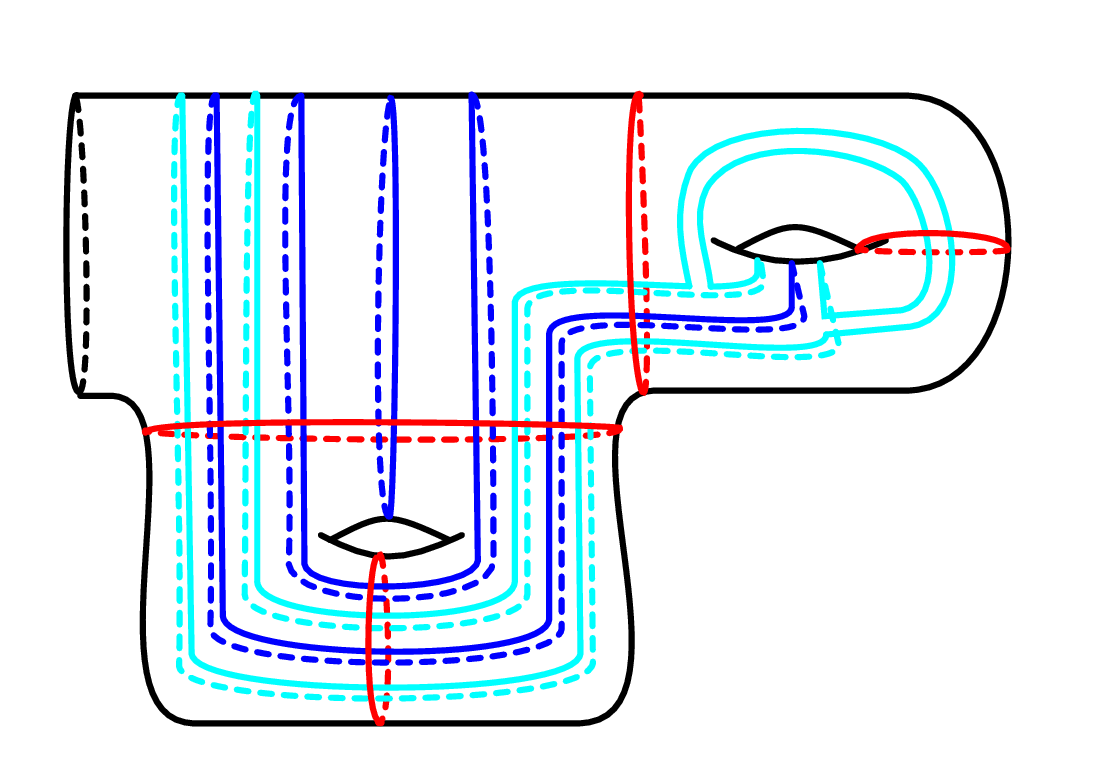
\caption{Two multicurves~$\alpha$ and~$\beta$ on the surface of genus two with 
one boundary component. {There are~$y$ parallel copies of the component~$\beta_1$, 
and the separating component~$\beta_3$ is drawn in light blue.}}
\label{genustwo}
\end{figure}
{We first note that the components~$\alpha_1,\beta_1,\alpha_3$ and~$\beta_3$ are separating. Furthermore,
the components~$\alpha_2$ and~$\alpha_4$ have their counterparts~$\beta_2$ and~$\beta_4$, respectively, 
in the multicurve~$\beta$ with which they each form a bounding pair.}
\smallskip

\noindent
We directly calculate 
\[ 
XX^\top = 
\begin{pmatrix}
84 + 16y & 40 + 8y & 40 & 16  \\
40 + 8y & 20 + 4y & 20 & 8 \\
40 & 20 & 20 & 8 \\
16 & 8 & 8 & 4
\end{pmatrix},
\]
and it is a direct check (by the computer) that the characteristic polynomial of~$XX^\top$ 
is irreducible if~$y=2$ or~$y=3$. This finishes the case~$g=2$ with one boundary 
component. 
\smallskip

\subsubsection{For~$g> 2$ and one boundary component} 
In order to increase the genus by one, we glue a surface of genus one with two boundary 
components as follows. On this surface, we consider the two multicurves~$\alpha$ 
and~$\beta$ shown in~\Cref{genusoneTorelli}. 
\begin{figure}[h]
\def\svgwidth{150pt}
\begingroup%
  \makeatletter%
  \providecommand\color[2][]{%
    \errmessage{(Inkscape) Color is used for the text in Inkscape, but the package 'color.sty' is not loaded}%
    \renewcommand\color[2][]{}%
  }%
  \providecommand\transparent[1]{%
    \errmessage{(Inkscape) Transparency is used (non-zero) for the text in Inkscape, but the package 'transparent.sty' is not loaded}%
    \renewcommand\transparent[1]{}%
  }%
  \providecommand\rotatebox[2]{#2}%
  \newcommand*\fsize{\dimexpr\f@size pt\relax}%
  \newcommand*\lineheight[1]{\fontsize{\fsize}{#1\fsize}\selectfont}%
  \ifx\svgwidth\undefined%
    \setlength{\unitlength}{241.86132048bp}%
    \ifx\svgscale\undefined%
      \relax%
    \else%
      \setlength{\unitlength}{\unitlength * \real{\svgscale}}%
    \fi%
  \else%
    \setlength{\unitlength}{\svgwidth}%
  \fi%
  \global\let\svgwidth\undefined%
  \global\let\svgscale\undefined%
  \makeatother%
  \begin{picture}(1,0.9516255)%
    \lineheight{1}%
    \setlength\tabcolsep{0pt}%
    \put(0,0){\includegraphics[width=\unitlength]{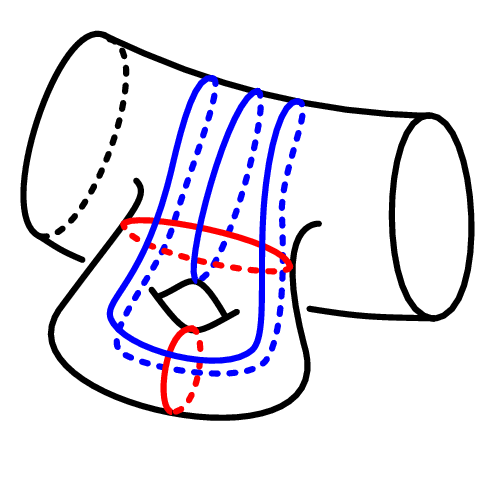}}%
    \put(0.64827827,0.42707547){\makebox(0,0)[lt]{\lineheight{1.25}\smash{\begin{tabular}[t]{l}$\alpha_1$\end{tabular}}}}%
    \put(0.30755527,0.05690556){\makebox(0,0)[lt]{\lineheight{1.25}\smash{\begin{tabular}[t]{l}$\alpha_2$\end{tabular}}}}%
    \put(0.62724389,0.799347){\makebox(0,0)[lt]{\lineheight{1.25}\smash{\begin{tabular}[t]{l}$y$ copies\end{tabular}}}}%
  \end{picture}%
\endgroup%

\caption{Two multicurves~$\alpha$ (in red) and~$\beta$ (in blue) on the surface of genus 
one with two boundary components. The multicurve~$\beta$ has~$y$ parallel copies 
of its separating component.}
\label{genusoneTorelli} 
\end{figure}
We directly calculate 
\[ 
XX^\top = 
\begin{pmatrix}
16y+4 & 8y \\
8y & 4y
\end{pmatrix} =: C_{y},
\]
and~$\chi_{XX^\top}(t) = t^2 -(20y+4)t+16y$ with discriminant~$16\cdot (25y^2+6y+1)$, 
which is never a square. Indeed, we have
\[
(5y)^2 = 25y^2 <  25y^2+6y+1 < 25y^2 +10y + 1 = (5y+1)^2.
\]
In particular, the polynomial~$\chi_{XX^\top}$ is irreducible for every positive integer~$y$.
\smallskip

\noindent
For the inductive step, let~$g \ge 2$. Assume we have constructed on the surface of 
genus~$g$ with one boundary component a pair of multicurves~$\alpha', \beta'$ 
such that the characteristic polynomial~$\chi_{XX^\top} \in \mathbb{Z}[t]$ is irreducible 
and of degree~$3g-2$, in such a way that 
the multicurves~$\alpha$ and~$\beta$ consist of components 
that are separating or that come in bounding pairs, where for each bounding pair one 
of the curves is a component of~$\alpha$ and the other is a component of~$\beta$. 
Further, assume that~$\alpha_1'$ is a simple closed curve that encircles all the handles 
of the surface, except for the rightmost one. Then, we take such a model surface and glue to 
its boundary a surface of genus one with two boundary components, as shown in 
\Cref{genusoneTorelli}, and add two new curves~$\alpha_0$ and~$\beta_0$ to the 
multicurves. The curve~$\alpha_0$ encircles all the handles of the newly formed surface, 
except for the rightmost one, and the curve~$\beta_0$ runs along the glued boundary 
components, and twice intersects~$\alpha_0$ but no other component of~$\alpha$, 
see~\Cref{torelli_induction}. 
\smallskip

\begin{figure}[h]
\def\svgwidth{300pt}
\begingroup%
  \makeatletter%
  \providecommand\color[2][]{%
    \errmessage{(Inkscape) Color is used for the text in Inkscape, but the package 'color.sty' is not loaded}%
    \renewcommand\color[2][]{}%
  }%
  \providecommand\transparent[1]{%
    \errmessage{(Inkscape) Transparency is used (non-zero) for the text in Inkscape, but the package 'transparent.sty' is not loaded}%
    \renewcommand\transparent[1]{}%
  }%
  \providecommand\rotatebox[2]{#2}%
  \newcommand*\fsize{\dimexpr\f@size pt\relax}%
  \newcommand*\lineheight[1]{\fontsize{\fsize}{#1\fsize}\selectfont}%
  \ifx\svgwidth\undefined%
    \setlength{\unitlength}{479.51716683bp}%
    \ifx\svgscale\undefined%
      \relax%
    \else%
      \setlength{\unitlength}{\unitlength * \real{\svgscale}}%
    \fi%
  \else%
    \setlength{\unitlength}{\svgwidth}%
  \fi%
  \global\let\svgwidth\undefined%
  \global\let\svgscale\undefined%
  \makeatother%
  \begin{picture}(1,0.3178559)%
    \lineheight{1}%
    \setlength\tabcolsep{0pt}%
    \put(0,0){\includegraphics[width=\unitlength]{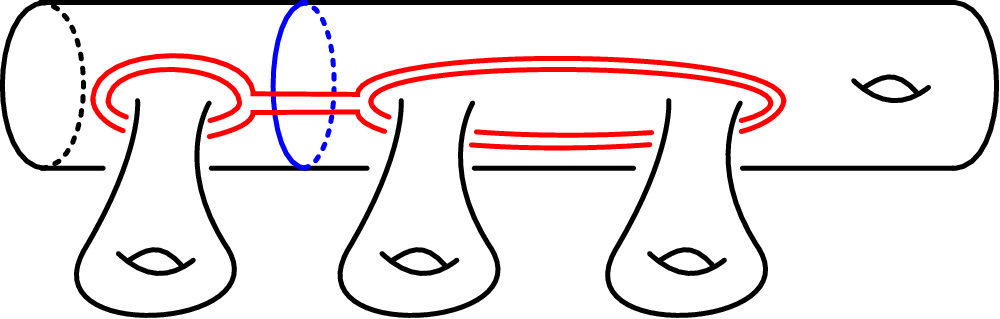}}%
    \put(0.51561118,0.21098752){\makebox(0,0)[lt]{\lineheight{1.25}\smash{\begin{tabular}[t]{l}$\cdots$\end{tabular}}}}%
    \put(0.58755856,0.20160308){\makebox(0,0)[lt]{\lineheight{1.25}\smash{\begin{tabular}[t]{l}$\alpha_1'$\end{tabular}}}}%
    \put(0.7251971,0.26103787){\makebox(0,0)[lt]{\lineheight{1.25}\smash{\begin{tabular}[t]{l}$\alpha_0$\end{tabular}}}}%
    \put(0.23720599,0.2766786){\makebox(0,0)[lt]{\lineheight{1.25}\smash{\begin{tabular}[t]{l}$\beta_0$\end{tabular}}}}%
    \put(0.15109441,0.21210001){\makebox(0,0)[lt]{\lineheight{1.25}\smash{\begin{tabular}[t]{l}$\alpha''_1$\end{tabular}}}}%
  \end{picture}%
\endgroup%

\caption{}
\label{torelli_induction}
\end{figure}

\noindent
The proof of irreducibility is now exactly the same as in the non-Torelli case in~\Cref{realisation_section}. 
The only thing we need to check is that the multicurves~$\alpha$ and~$\beta$ consist of components 
that are separating or that come in bounding pairs, where for each bounding pair one 
of the curves is a component of~$\alpha$ and the other is a component of~$\beta$. 
But this is clearly the case, since all the curves we add in the inductive step are 
separating or come as a bounding pair. 
\smallskip

\subsubsection{The closed case for~$g\ge4$.}
The last step is to make the surfaces closed. We simply glue together 
two pieces of genera~$g',g''$, where~$g'+g'' = g$, and one boundary component 
together along their boundaries. The same argument as in the inductive step provides 
irreducible characteristic polynomials of degree~$$3g' - 2 + 3g'' - 2 + 1 = 3g-3.$$ 

\subsubsection{The closed case for~$g=3$.}
We need a different argument. In this case, we start with the surface 
of genus two and one boundary component depicted in~\Cref{genustwo}, and close it off to the left by glueing 
a surface of genus one with one boundary component, see~\Cref{genus3}. 
\begin{figure}[h]
\def\svgwidth{400pt}
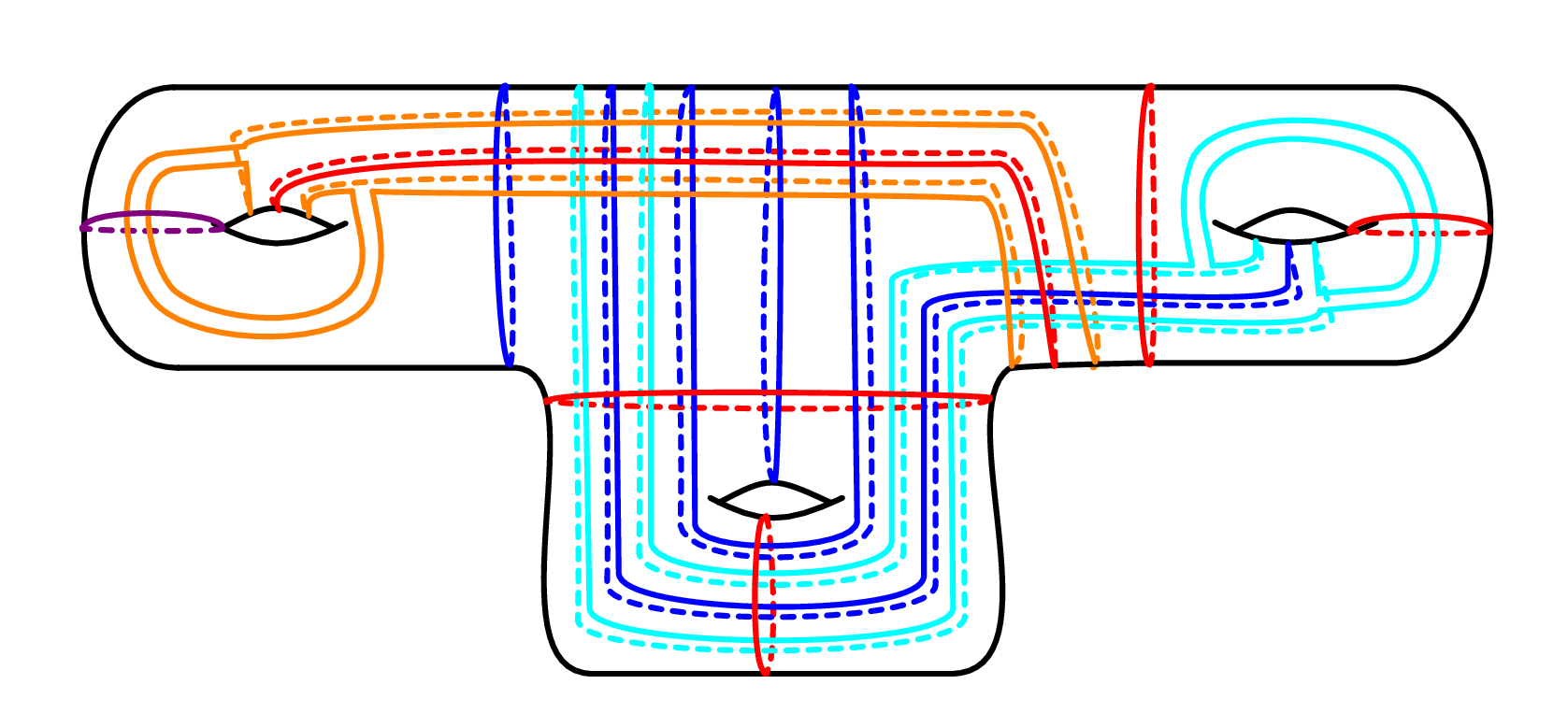
\caption{Two multicurves~$\alpha$ and~$\beta$ on the surface of genus three. 
There are to new components of~$\alpha$ when compared to~\Cref{genustwo}: 
a nonseparating component (red) that we call~$\alpha_5$ and a separating component
(orange) that we call~$\alpha_6$. Similarly, there are two new components of~$\beta$: 
a separating component (blue) that we call~$\beta_5$ and a nonseparating component 
({violet}) that we call~$\beta_6$.} 
\label{genus3}
\end{figure}
{We add the curves~$\alpha_5,\alpha_6$, two parallel copies of~$\beta_5$ and finally 
the curve~$\beta_6$. We note that~$\beta_5$ and~$\alpha_6$ are separating, and~$\beta_6$ and~$\alpha_5$ form a bounding pair.
To read off the matrix~$XX^\top$ for this pair of multicurves, we note 
two things. Firstly, the component~$\alpha_5$ intersects the curves~$\beta_j$ the same number of times as~$\alpha_1$, 
except for the component~$\beta_5$ of which there are two parallel copies. Secondly, the component~$\alpha_6$ 
intersects each curve~$\beta_j$ twice the number of times as~$\alpha_5$, except for~$\beta_6$. Using this, it is straightforward to calculate 
(ordering the curves as~$\alpha_6,\alpha_5,\alpha_1,\alpha_2,\alpha_3,\alpha_4$)
\[ 
XX^\top = 
\begin{pmatrix}
500 & 248 & 232 & 112 & 80 & 32\\
248 & 124 & 116 & 56 & 40 & 16 \\
232 & 116 & 116 & 56 & 40 & 16  \\
112 & 56 & 56 & 28 & 20 & 8 \\
80 & 40 & 40 & 20 & 20 & 8 \\
32 & 16 & 16 & 8 & 8 & 4
\end{pmatrix},
\]
which is then checked to have irreducible characteristic polynomial.}
\smallskip

\subsection{Multicurve intersection degrees~$d<3g-3$} We now show how to 
modify our construction from~\Cref{maximal_Torelli} in order to realise multicurve 
intersection degrees smaller than the maximal multicurve intersection degree~$3g-3$. 
We need new building blocks to construct our surfaces. 
\smallskip

\noindent
\emph{Block 1}. Our first block is obtained from the surface depicted in 
\Cref{genustwo}, simply by dropping the component~$\alpha_3$. 
A direct verification yields that for~$y=1,2$ the characteristic polynomial 
of~$XX^\top$ is irreducible and of degree 3.
\smallskip

\noindent
\emph{Block 2}. Our second block is obtained from the surface depicted in 
\Cref{D2}. The characteristic polynomial of the matrix~$XX^\top$ for the 
multicurves~$\alpha$ and~$\beta$ is irreducible and 
of degree 1. Versions of this block with distinct characteristic polynomial 
can be obtained by taking~$y$ parallel copies of~$\beta$. 
\begin{figure}[h]
\def\svgwidth{200pt}
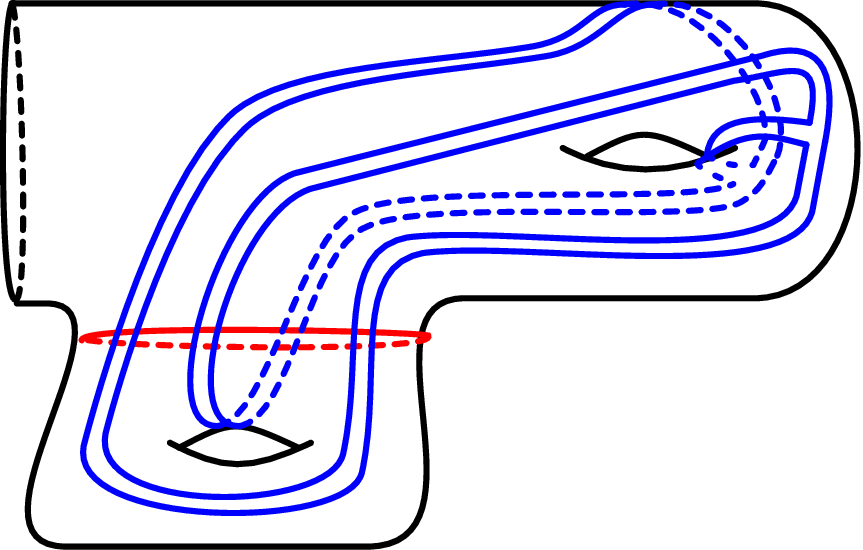
\caption{Two separating and filling {curves~$\alpha$ (red) and~$\beta$ (blue)} on the surface of genus two 
with one boundary component.}
\label{D2}
\end{figure}
\smallskip

\noindent
\emph{Block 3}. Take a surface as depicted in~\Cref{k_block_Torelli}.
We denote the red multicurve by~$\alpha$ and the blue multicurve by~$\beta$. 
The multicurve~$\alpha$ has~$k+1$ separating components: one for each of the handles that 
separates the handle, and one that separates all the handles.
We denote the component of~$\alpha$ that separates   
all the handles of the surface in~\Cref{k_block_Torelli} by~$\alpha_1$, and we denote the other separating 
components of~$\alpha$ {by~$\alpha_2, \alpha_4, \dots, \alpha_{2k}$} from left to right. 
Finally, the remaining nonseparating components of~$\alpha$ {are~$\alpha_3, \alpha_5, \dots, \alpha_{2k+1}$ } 
from left to right.
\begin{figure}[h]
\def\svgwidth{280pt}
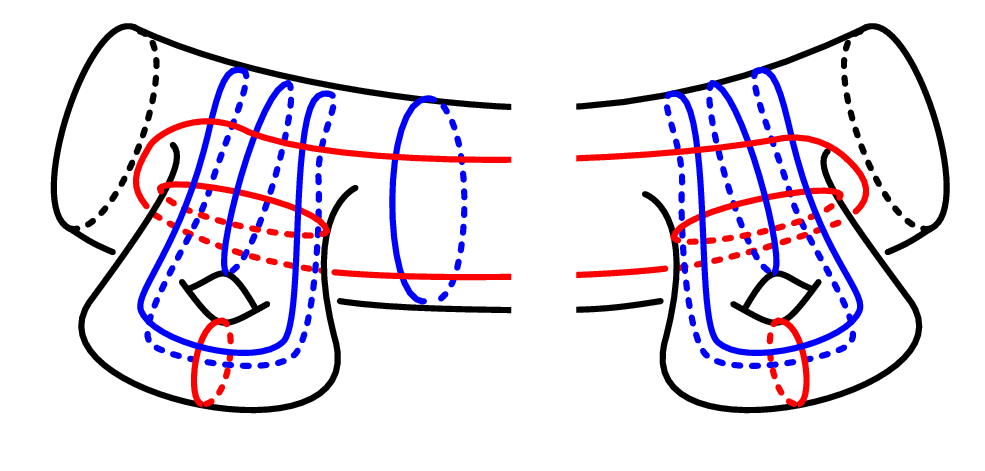
\caption{A surface of genus~$k$ with two boundary components, as well as  
two multicurves~$\alpha$ (in red) and~$\beta$ (in blue). 
The separating components of~$\beta$ can have several parallel copies: 
the {separating components~$\beta_1,\dots,\beta_k$ running through} the handles have~$y_1, \dots, y_k$ copies, and 
the separating {component~$\beta_0$ to the right of the left-most handle has~$y^2 - k - 4( y_1 + \cdots + y_k)$ copies.}}
\label{k_block_Torelli}
\end{figure}
{The multicurve~$\beta$ consists of a separating and a nonseparating component in each handle, 
as well as one separating curve~$\beta_0$ to the right of the left-most handle. All the separating components of~$\beta$ can have multiple parallel copies, 
see~\Cref{k_block_Torelli}.} 
In this situation, we have
\[ 
XX^\top =  \left(\begin{array}{c|ccccccc}
4y^2 & v_{y_1}^\top & v_{y_2}^\top & \cdots & v_{y_k}^\top \\
\hline
v_{y_1} & C_{y_1} & 0 & &  \\
v_{y_2} & 0 & C_{y_2} & &  \\
\vdots  & &  & \ddots &  \\
v_{y_k} & & & & C_{y_k} 
\end{array}\right),~C_{y_i} = \begin{pmatrix} 16y_i + 4& 8y_i \\ 8y_i & 4y_i \end{pmatrix},~v_{y_i} = 
\begin{pmatrix} 16y_i + 4 \\ 8y_i \end{pmatrix}.
\]

\noindent
If all the~$y_i$ are chosen pairwise distinct, ~\Cref{k_blocks_remark} guarantees that the polynomial~$\chi_{XX^\top}(t,y) \in \mathbb{Z}[t][y]$ 
is irreducible. By Hilbert's irreducibility theorem, there are infinitely many specifications 
of~$y$ such {that~$y^2 - k -4( y_1 + \cdots +y_k)>0$}
and such that~$\chi_{XX^\top}(t) \in \mathbb{Z}[t]$ 
is irreducible and of degree~$2k+1$.
Note that we can drop the separating components {$\alpha_2, \alpha_4, \dots, \alpha_{2k}$} of~$\alpha$ 
winding around one handle one by one in order to decrease the degree, reducing 
a 2-by-2 block to a 1-by-1 block, consisting of the coefficient~$4y_i$, for each component dropped in this way. 
Again, if all the~$y_i$ are chosen pairwise distinct,~\Cref{k_blocks_remark} 
guarantees that the polynomial~$\chi_{XX^\top}(t,y) \in \mathbb{Z}[t][y]$ is irreducible. 
 We can in this way construct all degrees~$k+1 \le d \le 2k+1$ for the 
surface of genus~$k$ and~2 boundary components. 

\subsubsection{Realising multicurve intersection degrees~$3g-6 \le d <3g-3$.}
Using a block of type 1 or 2 instead of our standard starting surface depicted in 
\Cref{genustwo},  we can reduce the multicurve intersection degree by 1 or 3, 
respectively. Since we use such a block on both sides of the surface in our 
construction, this gives the possibility to reduce 
the degree by any among the numbers 1,2,3,4 or 6. In particular, we can 
clearly realise the multicurve intersection {degrees~$3g-4,3g-5$ and~$3g-6$.}
This argument works for~$g\ge 4$. 
\smallskip

\noindent
In case of~$g=3$, we need a separate argument. The idea is to copy our example 
of maximal degree from~\Cref{genus3}, but leave out {first~$\alpha_1$ and then 
also~$\alpha_3$. 
If we drop~$\alpha_1$, which amounts do deleting the third row and column, we obtain a new matrix~$XX^\top$
which is directly verified to have an irreducible characteristic polynomial. If we now drop also~$\alpha_3$, 
which amounts to deleting the second-last row and column, we obtain a new matrix~$XX^\top$
which again is verified to have irreducible characteristic polynomial. Since we have only dropped separating components, 
we have not changed the fact that the multicurves~$\alpha$ and~$\beta$ consist of components 
that are separating or that come in bounding pairs, where for each bounding pair one 
of the curves is a component of~$\alpha$ and the other is a component of~$\beta$.} 
\smallskip

\noindent
We have therefore realised~$d=5,4$ for~$g=3$. For~$g=3$, the case~$d=3$ equals the case~$d=3g-6$, 
which is treated below.

\subsubsection{Realising multicurve intersection degrees~$g \le d \le 3g-6$.}
We start by constructing a surface of genus~$g-2$ with two boundary 
components, which we then close off in a second step.
\smallskip

\noindent
Using surfaces of the type depicted in~\Cref{genusoneTorelli} and 
applying the inductive step procedure illustrated in~\Cref{power2induction}, we can construct a surface of genus~$g-2\ge1 $ 
and two boundary components, as well as filling multicurves~$\alpha$ and~$\beta$ 
with intersection degree~$3(g-2) - 1 = 3g-7$. {The count of this intersection degree is exactly as in the 
non-Torelli case of~\Cref{realisation_section}, where we obtained degree~$3g-1$ 
on the closed surface of genus~$g$ with 2 boundary components}.
Using at some point in the inductive procedure a  block of type 3 of genus~$k \le g-2$, 
as depicted in~\Cref{k_block_Torelli},
we can reduce the degree by up to~$2k-2 \le 2g-6$, realising multicurve intersection degrees 
from~$g-1$ to~$3g-7$ on the surface of genus~$g-2$ with two boundary components. 
Now we close the surface, as depicted in~\Cref{closedTorelli}, adding the new 
components~$\alpha_0$ and~$\beta_0$ to the multicurves~$\alpha$ and~$\beta$, respectively.

\begin{figure}[h]
\def\svgwidth{400pt}
\begingroup%
  \makeatletter%
  \providecommand\color[2][]{%
    \errmessage{(Inkscape) Color is used for the text in Inkscape, but the package 'color.sty' is not loaded}%
    \renewcommand\color[2][]{}%
  }%
  \providecommand\transparent[1]{%
    \errmessage{(Inkscape) Transparency is used (non-zero) for the text in Inkscape, but the package 'transparent.sty' is not loaded}%
    \renewcommand\transparent[1]{}%
  }%
  \providecommand\rotatebox[2]{#2}%
  \newcommand*\fsize{\dimexpr\f@size pt\relax}%
  \newcommand*\lineheight[1]{\fontsize{\fsize}{#1\fsize}\selectfont}%
  \ifx\svgwidth\undefined%
    \setlength{\unitlength}{750.34232701bp}%
    \ifx\svgscale\undefined%
      \relax%
    \else%
      \setlength{\unitlength}{\unitlength * \real{\svgscale}}%
    \fi%
  \else%
    \setlength{\unitlength}{\svgwidth}%
  \fi%
  \global\let\svgwidth\undefined%
  \global\let\svgscale\undefined%
  \makeatother%
  \begin{picture}(1,0.37397243)%
    \lineheight{1}%
    \setlength\tabcolsep{0pt}%
    \put(0,0){\includegraphics[width=\unitlength]{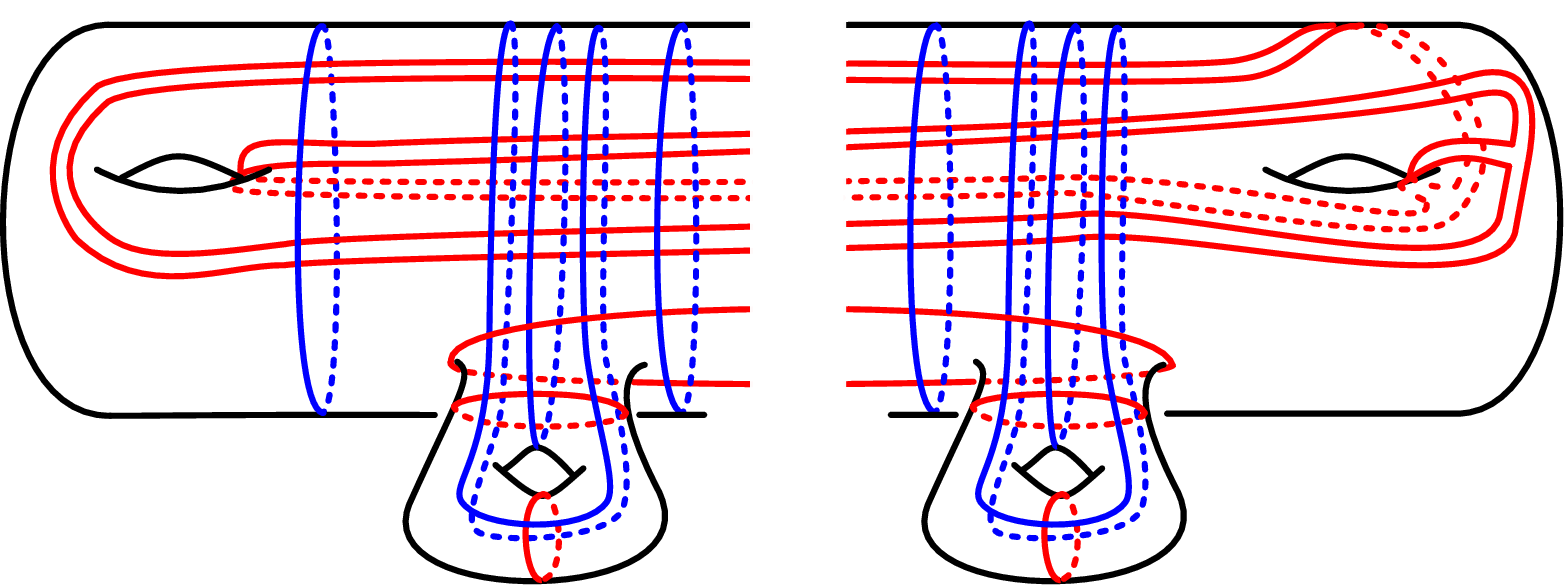}}%
    \put(0.49546459,0.04662666){\makebox(0,0)[lt]{\lineheight{1.25}\smash{\begin{tabular}[t]{l}$\cdots$\end{tabular}}}}%
    \put(0.15838798,0.13302278){\makebox(0,0)[lt]{\lineheight{1.25}\smash{\begin{tabular}[t]{l}$\beta_0$\end{tabular}}}}%
    \put(0.02970062,0.19365721){\makebox(0,0)[lt]{\lineheight{1.25}\smash{\begin{tabular}[t]{l}$\alpha_0$\end{tabular}}}}%
  \end{picture}%
\endgroup%

\caption{Two separating curves~$\alpha_0$ and~$\beta_0$. There are~$\rho$ parallel copies of~$\beta_0$.}
\label{closedTorelli}
\end{figure}

\noindent
We obtain the matrix
\[
XX^\top=
\left(
\begin{array}{c|ccc}
64\rho + 16a_1 & 4a_1 & \dots & 4a_n \\
\hline 
4a_1 &  &  &  \\
\vdots  &  & A &   \\
4a_n &  &  & 
\end{array}
\right), 
\]
where~$A$ is the matrix~$XX^\top$ before adding the curves~$\alpha_0$ and~$\beta_0$. 
Since~$a_1 = 4a$, we can set~$\rho = y^2-a$ to have the top-left coefficient~$64y^2$, which is 
exactly the form of the matrix in~\Cref{add_one_genus_lemma}. Finishing the argument 
as usual, we can realise the multicurve intersection degrees~$g\le d \le 3g-6$ for~$g\ge 3$. 

\subsubsection{Realising multicurve intersection degrees~$1 \le d <g$.}
Realising multicurve intersection degree one is clearly achieved by taking a pair 
of separating filling curves on the surface~$S_g$. 
\smallskip

\noindent
For~$2\le d< g$, let us define~$f=g-1-d$.  
We start with a surface block of type~3 of genus~$g-2$, where we deleted all the 
components {$\alpha_2, \alpha_4, \dots, \alpha_{2k}$} of~$\alpha$. 
We also remove the component~{$\beta_0$}  of~$\beta$, 
see~\Cref{k_block_Torelli}.  Furthermore, we let the~$f + 1\le g-2$ first
of the parameters~$y_i$ be equal to~$1$. Then we close off the surface as in the previous 
case, adding one component~$\alpha_0$ to~$\alpha$ and one component~$\beta_0$ to~$\beta$, 
compare with~\Cref{closedTorelli}. Assume 
there are~$\rho$ parallel copies of~$\beta_0$. 
We get
\[
XX^\top =  \begin{pmatrix}
64(\rho-g+2) + 256\delta & 32y_1 & 32y_2 & \cdots & 32y_{g-2} \\
32y_1 & 4y_1 & &&  \\
32y_2 &  & 4y_2 & &  \\
\vdots  & &  & \ddots &   \\
32y_{g-2} & & & & 4y_{g-2} 
\end{pmatrix}, 
\]
where~$\delta = y_1+ \cdots + y_{g-2}$.
We choose~$\rho$ such that~$64(\rho-g+2) + 256\delta = 64y^2$. 
To simplify the calculations, we let~$z_i = 4y_i$ for~$i=1,\dots,g-2$.
The matrix becomes
\[
XX^\top =  \begin{pmatrix}
64y^2 & 8z_1 & 8z_2 & \cdots & 8z_{g-2}  \\
8z_1 & z_1 & &&  \\
8z_2 &  & z_2 & &  \\
\vdots  & &  & \ddots &   \\
8z_{g-2} & & & & z_{g-2} 
\end{pmatrix}.
\]
By Lemma~9 in~\cite{LL}, the characteristic polynomial of~$XX^\top$ equals
\[
p(t,y,\mathbf{z})=
-64y^2\prod_{i=1}^{g-2}(t-z_i) + t\prod_{i=1}^{g-2}(t-z_i) - \sum_{i=1}^{g-2} 64z_i^2\prod_{j\ne i}(t-z_j).
\]
If all the~$z_i$ are pairwise distinct, this polynomial is irreducible as a polynomial in~$t,y$ 
by Lemma~10 in~\cite{LL}. However, we chose that the first~$f+1$ coefficients~$y_1, \dots, y_{f+1}$ 
are equal to~$1$ and the other~$y_i \ne 1$ and pairwise distinct. 
In particular, the polynomial~$p(t,y)$ factors as~$(t-4)^{f}\tilde{p}(t,y)$, where~$\tilde{p}(t,y)$ is of degree~$g-1-f =d$ 
in the variable~$t$ and with pairwise distinct~$z_i$. In particular, Lemma~10 in~\cite{LL} 
implies that~$\tilde{p}(t,y) \in \mathbb{Z}[t,y]$ is irreducible. Hilbert's irreducibility theorem 
guarantees the existence of infinitely many positive specifications of~$y$ such that the resulting 
polynomial is irreducible in~$\mathbb{Z}[t]$. 
\smallskip

\noindent
This finishes the proof of~\Cref{thm:main:1}. 
\smallskip

\noindent
Finally, we end this section with a proof of~\Cref{multicurvedegreethm:bipartite}.

\begin{proof}[Proof of~\Cref{multicurvedegreethm:bipartite}]
For every~$g\ge 3$ and every integer~$1 \le d\leq 3g-3$, we have constructed a filling pair of 
multicurves~$\alpha$ and~$\beta$, with a parameter~$y$, such that~$\chi_{XX^\top}(t,y)\in\mathbb{Z}[t,y]$ 
is irreducible. By~\Cref{rk:root}, we may run the same argument to show that also the 
polynomial~$\chi_{XX^\top}(t^2,y)\in\mathbb{Z}[t,y]$ is irreducible. By Hilbert's irreducibility 
theorem, we find infinitely many specifications of~$y$ such that~$\chi_{XX^\top}(t^2)\in\mathbb{Z}[t]$ 
is irreducible of degree~$2d$. 
The leading eigenvalue~$\mu$ of $XX^\top$ is a root of a characteristic polynomial~$\chi_{XX^\top}(t)$, 
so~$\chi_{XX^\top}(t^2)\in\mathbb{Z}[t]$ is the minimal polynomial of~$\sqrt{\mu}$. Hence, 
the multicurve bipartite degree of~$\alpha$ and~$\beta$ equals~$2d$. 
\smallskip

\noindent
For~$g=2$, we use the example constructed in~\Cref{realisation_section} for~$d=3$ and the examples 
in~\cite{LL} for~$d=1,2$. Similarly to~\Cref{rk:root}, one can run the same proof 
as~\cite[Lemma 10]{LL} to show that~$\chi_{XX^\top}(t^2,y)\in\mathbb{Z}[t,y]$ is irreducible.
\end{proof}

\section{Explicit pseudo-Anosov maps}
\label{sec:proof}

\noindent
{There are two goals of this section, both inspired by a problem posed by Margalit~\cite[Problem 10.4]{Margalit:question}.
This section is completely independent of the proof of our main results, \Cref{thm:trace}, \Cref{thm:stretch} and~\Cref{thm:main:1}}. 
\smallskip

\noindent
{The first goal is to construct, as explicitly as possible, pseudo-Anosov maps 
whose stretch factors have prescribed algebraic degrees. For suitably chosen examples, we can 
replace our nonsplitting criterion,~\Cref{thm:criterion:nonsplitting}, by the nonsplitting criterion of~\cite[Theorem~6]{LL}. 
This provides that the mapping class~$T_\alpha \circ T_\beta$ has degree two over the trace field, which 
is more explicit than what we obtain using~\Cref{thm:criterion:nonsplitting}. However, we are still not able to 
control the exact number of parallel copies used for some of the components~$\beta_i$. 
This construction is carried out in~\Cref{sec:smalldegrees} and~\Cref{sec:even_deg_almost_explicit}.}
\smallskip

\noindent
{The second goal is to construct completely explicit pseudo-Anosov mapping classes 
of the form~$T_\alpha \circ T_\beta$ with maximal stretch factor degree~$6g-6$. 
We are able to achieve this for~$g\le 201$ in~\Cref{sec:completely_explicit}.}

\noindent
\subsection{Multicurve intersection degrees~$d \le 3g-3$}
\label{sec:smalldegrees}
{Recall that in~\Cref{realisation_section}, we constructed pairs of multicurves for which the 
Thurston--Veech construction realises the multicurve intersection degree~$3g-3$, for~$g\ge 2$. }
In a first step, we build upon our examples from~\Cref{realisation_section} to realise also all possible smaller multicurve intersection
degrees~$1 \le d <3g-3$ on~$S_g$ for~$g\ge 2$. 
\smallskip

\noindent
We need a new building block for our surfaces, which is depicted in~\Cref{k_block}. 
\smallskip

\begin{figure}[h]
\def\svgwidth{220pt}
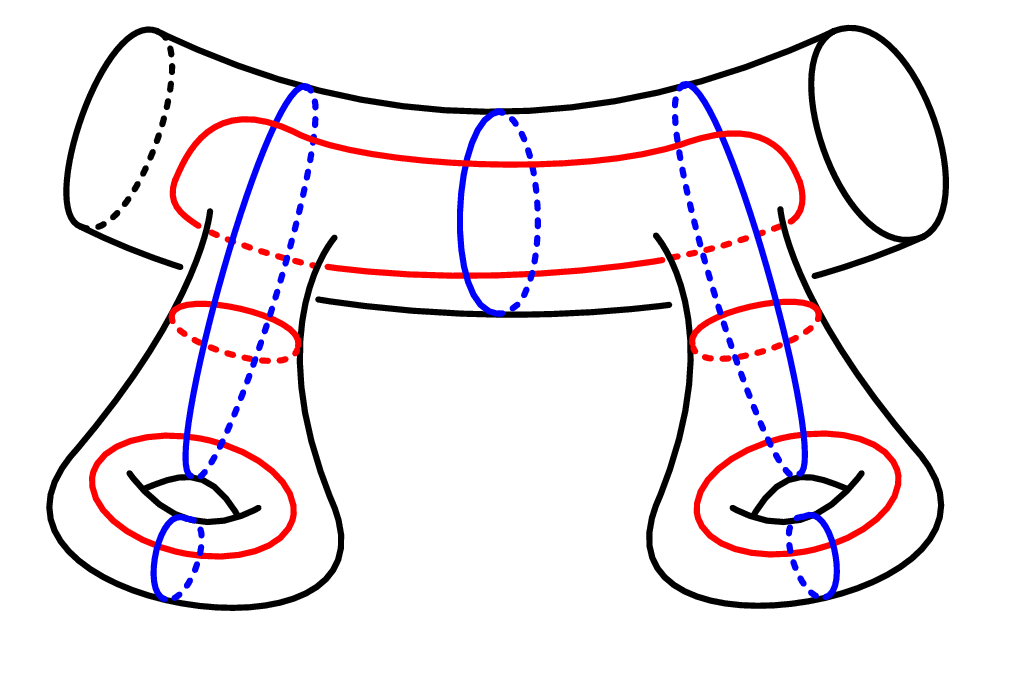
\caption{A surface of genus~$k$ with two boundary components, as well as  
two multicurves~$\alpha$ (in red) and~$\beta$ (in blue). 
Some components of~$\beta$ have several parallel copies, as indicated by~$y_1,\dots,y_k$ and~$y^2-k$.}
\label{k_block}
\end{figure}

\noindent
We denote the red multicurve by~$\alpha$ and the blue multicurve by~$\beta$. 
The multicurve~$\alpha$ has~$k+1$ separating components: one for each of the handles that 
separates the handle, and one that separates all the handles. We denote the component of~$\alpha$ that separates   
all the handles of the surface in~\Cref{k_block} by~$\alpha_1$, and we denote the other separating 
components of~$\alpha$ {by~$\alpha_2, \alpha_4, \dots, \alpha_{2k}$} from left to right. 
Finally, the remaining nonseparating components of~$\alpha$ {are~$\alpha_3, \alpha_5, \dots, \alpha_{2k+1}$ }
from left to right.
\smallskip

\noindent
In this situation, we have 

\[
XX^\top =  \left(\begin{array}{c|ccccccc}
4y^2 & v^\top & v^\top & \cdots & v^\top \\
\hline
v & B_{y_1} & 0 & &  \\
v & 0 & B_{y_2} & &  \\
\vdots  & &  & \ddots &  \\
v & & & & B_{y_k} 
\end{array}\right), \qquad B_{y_i} = \begin{pmatrix} 4 & 2 \\ 2 & y_i \end{pmatrix}, \qquad v = \begin{pmatrix} 4 \\ 2 \end{pmatrix}.
\]

\noindent
Let~$p_{y_i}(t) = t^2 - (4+y_i)t + 4(y_i-1)$ be the characteristic polynomial 
of~$B_{y_i}$. We know from~\Cref{g1_section} 
that~$p_{y_i}$ is irreducible if~$y\ge 12$. 
So, choosing all~$y_i\ge 12$ pairwise distinct,~\Cref{k_blocks_remark} 
guarantees that the polynomial~$\chi_{XX^\top}(t,y) \in \mathbb{Z}[t][y]$ is irreducible. By Hilbert's irreducibility 
theorem, there are infinitely many specifications of~$y$ such that~$y^2-k>0$ and 
such that~$\chi_{XX^\top}(t) \in \mathbb{Z}[t]$ is irreducible and of degree~$2k+1$.
\smallskip

\noindent
\emph{Case 1: $2g \le d< 3g-3$}. Assume we want to realise the multicurve intersection 
degree~$3g-3-f$ for~$0<f\le g-3$. Let~$k=f+2 \le g-1$. Start the inductive procedure as in 
\Cref{inductive_step} with the surface from~\Cref{k_block} as a starting point, 
adding~$g-1-k$ more handles. The exact same argument yields a surface of genus~$g-1$ 
with two boundary components, and a characteristic polynomial~$\chi_{XX^\top}\in\mathbb{Z}[t]$ 
that is irreducible and of degree~$2k+1 + 3(g-1-k) = 3g-3-k+1$. Closing up the surface exactly as in 
\Cref{closedcase} yields~$3g-3-k+2 = 3g-3-f$ as a multicurve intersection degree
on the closed orientable surface of genus~$g$. 
\smallskip

\noindent
\emph{Case 2: $g<d<2g$.} Assume we want to realise the multicurve intersection 
degree~$2g-f$ for~$0<f\le g-1$. Take the surface depicted in~\Cref{k_block} 
for~$k=g-1$. Now, remove~$f$ of the separating curve~$\alpha_2, \dots, \alpha_{2g-2}$. 
This slightly modifies the matrix~$XX^\top$:~$f$ of the 2-by-2 blocks on the diagonal 
are now 1-by-1 blocks, with the single coefficient~$y_i$. Nevertheless, since all the~$y_i$ 
are chosen pairwise distinct,~\Cref{k_blocks_remark} 
guarantees that the polynomial~$\chi_{XX^\top}(t,y) \in \mathbb{Z}[t][y]$ is irreducible. 
We note that for the coefficients~$y_i$ in the 1-by-1 blocks, any positive integer can be chosen. 
By Hilbert's irreducibility 
theorem, there are infinitely many specifications of~$y$ such that~$\chi_{XX^\top}(t) \in \mathbb{Z}[t]$ 
is irreducible and of degree~$2g-1-f$. Closing up the surface as in~\Cref{closedcase} yields 
the multicurve intersection degree~$2g-f$ on the closed orientable surface of genus~$g$. 
\smallskip

\noindent
\emph{Case 3: $1\le d \le g$}. This is the case we have already dealt with in~\cite{LL}.

\subsection{Even degree stretch factors}
\label{sec:even_deg_almost_explicit}

\noindent
In this section, we show that in our construction of multicurves in~\Cref{sec:smalldegrees}, 
the degree of the stretch factor of~$T_\alpha \circ T_\beta$ equals two over the trace field. 
It uses the nonsplitting criterion of~\cite[Theorem~6]{LL} that we recall below.

\begin{thm}[\cite{LL}, Theorem~6]
\label{thm:criterion}
Let~$\alpha,\beta \subset S$ be a pair of filling multicurves. Let~$X$ be their 
geometric intersection matrix, let~$d$ be their multicurve intersection degree and 
let~$\Omega = \left(\begin{smallmatrix}0 & X \\ X^\top & 0\end{smallmatrix}\right)$. 
If~$\dim(\Omega)>\sigma(\Omega + 2I) + \mathrm{null}(\Omega + 2I) >\dim(\Omega)-2d$, 
then the mapping class~$T_\alpha \circ T_\beta$ is pseudo-Anosov with stretch 
factor~$\lambda$ of degree~$2d$.
\end{thm}

\noindent
Here,~$\sigma(A)$ and~$\mathrm{null}(A)$ denote the signature and the nullity,
respectively, of the matrix~$A$. 

\begin{thm}
\label{sf_degrees}
Let~$\alpha$ and~$\beta$ be an example of a pair of multicurves described in~\Cref{sec:smalldegrees}, 
realising a multicurve intersection degree~$1\le d\le 3g-3$.  
Then the mapping class~$T_\alpha \circ T_\beta$ is pseudo-Anosov with stretch factor~$\lambda$ of degree $2d$.
\end{thm}

\noindent 
For the case~$1\le d \le g$, this is shown in~\cite{LL}. 

\begin{proof}[Proof of~\Cref{sf_degrees}]
According to~\Cref{thm:criterion}, all there is to show is
\begin{align}
\label{ineq}
\dim(\Omega) > \sigma(\Omega + 2I) + \mathrm{null}(\Omega + 2I) > \dim(\Omega) - 2d.
\end{align}

\noindent
We now make a case distinction depending on~$d$. 
\smallskip

\noindent
\emph{Case 1: $2g \le d \le 3g-3$}. 
We consider the submatrix~$\Omega'$ of~$\Omega$ that is obtained by deleting all 
the rows and columns corresponding to components of the multicurve~$\alpha$ that 
have been added during the inductive step or closing up of the surface. Furthermore, 
if~$d<3g-3$, we also remove the component of~$\alpha$ encircling multiple handles 
of the starting surface, that is, the component~$\alpha_1$ of the surface depicted in~\Cref{k_block}. 
\smallskip

\noindent
A base change by a permutation matrix brings~$\Omega' + 2I$ into block diagonal form 
with~$g-1$ blocks corresponding to genus one surface pieces as depicted in~\Cref{genusone}, 
and a block of the form~$2I$. For a block of the former type, and for~$y>4$, 
we directly calculate that the nullity is zero and the signature equals the dimension of the block 
minus two. Already, this implies that certainly the signature of~$\Omega+2I$ is not equal to its 
dimension, and it only remains to verify the lower bound in~\Cref{ineq}. 
\smallskip

\noindent
By construction, if the genus equals~$g \ge 2$, we have~$g-1$ surface pieces as in~\Cref{genusone}. 
This in particular implies that~$\sigma(\Omega') =\dim(\Omega') - 2g + 2$. Furthermore, 
we have~$\dim(\Omega) - \dim(\Omega') = d - 2g+2$. 
The latter equality follows from that fact that  the number of components of~$\alpha$ 
in our construction is exactly~$d$, 
and there are two components per surface pieces as in~\Cref{genusone}. 
We now calculate 
\begin{align*}
\sigma(\Omega+2I) &\ge \sigma(\Omega') - (\dim(\Omega) - \dim(\Omega')) \\
&= (\dim(\Omega') - 2g + 2) - (d - 2g+2) \\
&= \dim(\Omega') - d \\
&> \dim(\Omega) -2d,
\end{align*}
which implies~\Cref{ineq}, so we are done for this case.
\smallskip

\noindent
\emph{Case 2: $g<d < 2g$.}
We consider the submatrix~$\Omega'$ of~$\Omega$ that is obtained by deleting  
two rows and two columns corresponding to the following components of the multicurve~$\alpha$: 
the one corresponding to the component~$\alpha_1$ encircling multiple handles in~\Cref{k_block} 
and the one obtained from closing the surface. Recall that we have removed~$f = 2g-d$ 
of the separating curves~$\alpha_2, \dots, \alpha_{2g-2}$. 
\smallskip

\noindent
A base change by a permutation matrix brings~$\Omega' + 2I$ into block diagonal form 
with~$g-1- f$ blocks corresponding to surface pieces as in~\Cref{genusone},~$f$ 
blocks corresponding to surface pieces as in~\Cref{genusone} but with the separating 
component of~$\alpha$ removed, and a block of the form~$2I$. 
\smallskip

\noindent
For a block of the first type, and for~$y>4$, 
recall from the previous case that the nullity is zero and the signature equals the dimension of the block 
minus two. 
For a block of the second type, the sum of the nullity and the signature equals the dimension 
of the block if~$y \le 3$, and it equals the dimension of the block minus two if~$y>3$. 
We may assume that for at least one block of the second type, we have~$y = 3$. This 
is enough to ensure that~$\dim(\Omega) > \sigma(\Omega + 2I) + \mathrm{null}(\Omega + 2I)$, 
so again we only need to verify the lower bound in~\Cref{ineq}.
\smallskip

\noindent
By construction, if the genus equals~$g \ge 2$, we have~$g-1$ surface pieces as in~\Cref{genusone}. 
Having at least one piece with~$y\le 3$, this implies that~$\sigma(\Omega') > \dim(\Omega') - 2g + 2$. Furthermore, 
we have~$\dim(\Omega) - \dim(\Omega') = 2$. 
We now calculate 
\begin{align*}
\sigma(\Omega+2I) &\ge \sigma(\Omega') - (\dim(\Omega) - \dim(\Omega')) \\
&> (\dim(\Omega') - 2g + 2) - 2 \\
&= \dim(\Omega') -2g  \\
&= \dim(\Omega) -2g + 2 
\ge \dim(\Omega) - 2d,
\end{align*}
which implies~\Cref{ineq} also in the case~$g<d < 2g$, so we are done.
\end{proof}

\subsection{Explicit examples with stretch factor degree~$6g-6$}
\label{sec:completely_explicit}
We conclude this section by giving explicit computations, for~$1<g\le 201$, supporting a conjecture on the irreducibility 
of the characteristic polynomials constructed in~\Cref{realisation_section} for specific values of~$y$. 
In the inductive step of~\Cref{inductive_step}, one uses a map
\[
\phi_k : \begin{array}{cccc}
M_{k}(\mathbb Z)\times \mathbb Z&\longrightarrow &M_{k+3}(\mathbb Z) \\
(C,y)&\mapsto  &\left(
\begin{array}{c|c|c}
4y^2 & * & *  \\
\hline 
*  & C & \\
\hline
 * &  & A
\end{array}\right)
\end{array}, 
\qquad \mathrm{with} \ A=\left(
\begin{smallmatrix}
4 & 2   \\
2 & 12
\end{smallmatrix}
\right).
\]
For~$g>1$ we inductively construct the~$(3g-1) \times (3g-1)$ matrix~$M_g$ 
with the maps~$\phi_{3i-1}$ for~$i=1,\dots,g-1$:
\[
M_g = \phi_{3(g-1)-1}(\phi_{3(g-2)-1}(\dots \phi_{3\cdot 2 -1}(\phi_{3\cdot 1 -1}(B,y^{(1)}),y^{(2)}),\dots,y^{(g-2)}),y^{(g-1)}),
\]
with~$B=\left( \begin{smallmatrix}4 & 2   \\ 2 & 13 \end{smallmatrix}\right)$ and suitable 
parameters~$y^{(i)}$ given by Hilbert's irreducibility theorem. The condition~$y^2 > \frac1{4}c_{11} + 1$ 
appearing in the construction is obviously equivalent to~$(y^{(i)})^2 > (y^{(i-1)})^2 + 1$. 
Finally, following~\Cref{closedcase} the matrix~$XX^\top$ for the multicurves~$\alpha$ and~$\beta$ 
on the closed surface of genus~$g+1$ takes the form
\[
N_g =
\left(
\begin{array}{c|ccc}
y^2 & * \\
\hline 
* & M_g 
\end{array}
\right)
\]
with the condition~$y^2 > \frac1{4} (M_g)_{11} = (y^{(g-1)})^2$.
\smallskip

\noindent 
By computer {assistance~\cite{sage}}, one immediately checks the following proposition.

\begin{prop}
\label{prop:irr}
For any $1 < g \leq 200$, if $y^{(i)} = i+1$ for $i=1,\dots,g-1$, then the characteristic polynomial $\chi_{M_g}$ is irreducible over $\mathbb Q$. Moreover for 
$y = g+1$, $\chi_{N_g}$ is irreducible over $\mathbb Q$.
\end{prop}

\noindent
Together with~\Cref{sf_degrees}, this gives explicit examples of pseudo-Anosov maps 
realising the upper bound~$6g-6$ in~\Cref{thm:main:1} for every~$1 < g \leq 201$. 
We do not know whether~$\chi_{M_g}$ and~$\chi_{N_g}$ are actually irreducible for 
every~$g>200$ with the parameters~$y^{(i)}=i+1$ chosen as in~\Cref{prop:irr}.

\bibliographystyle{alphaurl}
\bibliography{biblio}

\newcommand{\etalchar}[1]{$^{#1}$}
\begin{thebibliography}{HMTY08}

\bibitem[{And}87]{Ando}
Tsuyoshi {Ando}.
\newblock {Totally positive matrices}.
\newblock {\em {Linear Algebra Appl.}}, 90:165--219, 1987.
\newblock \href {https://doi.org/10.1016/0024-3795(87)90313-2}
  {\path{doi:10.1016/0024-3795(87)90313-2}}.

\bibitem[Bel21]{flipper}
Mark Bell.
\newblock flipper (computer software).
\newblock \url{https://pypi.org/project/flipper/}, 2013--2021.
\newblock Version 2 July 2021.

\bibitem[Bou22]{Boulanger}
Julien Boulanger.
\newblock Central points of the double heptagon translation surface are not
  connection points.
\newblock {\em Bull. Soc. Math. Fr.}, 150(2):459--472, 2022.
\newblock \href {https://doi.org/10.24033/bsmf.2851}
  {\path{doi:10.24033/bsmf.2851}}.

\bibitem[Gan59]{Gantmacher}
F.~R. Gantmacher.
\newblock {\em The Theory of Matrices, Vol 2}.
\newblock AMS Chelsea Publishing, Providence, RI, 1959.

\bibitem[GJ00]{Gutkin:Judge}
Eugene {Gutkin} and Chris {Judge}.
\newblock {Affine mappings of translation surfaces: Geometry and arithmetic}.
\newblock {\em {Duke Math. J.}}, 103(2):191--213, 2000.
\newblock \href {https://doi.org/10.1215/S0012-7094-00-10321-3}
  {\path{doi:10.1215/S0012-7094-00-10321-3}}.

\bibitem[Hal33]{Marshall}
Marshall Hall.
\newblock Quadratic residues in factorization.
\newblock {\em Bull. Am. Math. Soc.}, 39:758--763, 1933.
\newblock \href {https://doi.org/10.1090/S0002-9904-1933-05730-0}
  {\path{doi:10.1090/S0002-9904-1933-05730-0}}.

\bibitem[HK06]{HironakaKin}
Eriko Hironaka and Eiko Kin.
\newblock {A family of pseudo-Anosov braids with small dilatation}.
\newblock {\em Algebraic \& Geometric Topology}, 6(2):699 -- 738, 2006.
\newblock \href {https://doi.org/10.2140/agt.2006.6.699}
  {\path{doi:10.2140/agt.2006.6.699}}.

\bibitem[HMTY08]{Hanson2008GeneralizedCF}
E.~H. Hanson, Adam~Benjamin Merberg, Christopher Towse, and Elena Yudovina.
\newblock Generalized continued fractions and orbits under the action of hecke
  triangle groups.
\newblock {\em Acta Arithmetica}, 134:337--348, 2008.
\newblock URL: \url{https://api.semanticscholar.org/CorpusID:58923865}.

\bibitem[KS00]{Kenyon:Smillie}
Richard {Kenyon} and John {Smillie}.
\newblock {Billiards on rational-angled triangles}.
\newblock {\em {Comment. Math. Helv.}}, 75(1):65--108, 2000.
\newblock \href {https://doi.org/10.1007/s000140050113}
  {\path{doi:10.1007/s000140050113}}.

\bibitem[LL24]{LL}
Erwan Lanneau and Livio Liechti.
\newblock Trace field degrees of abelian differentials.
\newblock {\em Comment. Math. Helv.}, 99(1):81--110, 2024.
\newblock \href {https://doi.org/10.4171/CMH/565} {\path{doi:10.4171/CMH/565}}.

\bibitem[Mar19]{Margalit:question}
D.~Margalit.
\newblock {Problems, questions, and conjectures about mapping class groups}.
\newblock In {\em {Breadth in contemporary topology}}, volume {102} of {\em
  {Proc. Sympos. Pure Math.}}, pages {157--186}. {Amer. Math. Soc., Providence,
  RI}, {2019}.

\bibitem[Mur08]{Ram}
M.~Ram Murty.
\newblock Polynomials assuming square values.
\newblock In {\em Number theory and discrete geometry. Proceedings of the
  international conference in honour of Professor R. P. Bambah, Chandigarh,
  India, November 30--December 3, 2005}, pages 155--163. Mysore: Ramanujan
  Mathematical Society, 2008.

\bibitem[{Str}17]{Strenner:algebraic}
Bal\'azs {Strenner}.
\newblock {Algebraic degrees of pseudo-Anosov stretch factors}.
\newblock {\em {Geom. Funct. Anal.}}, 27(6):1497--1539, 2017.
\newblock \href {https://doi.org/10.1007/s00039-017-0429-4}
  {\path{doi:10.1007/s00039-017-0429-4}}.

\bibitem[S{\etalchar{+}}YY]{sage}
W.\thinspace{}A. Stein et~al.
\newblock {\em {S}age {M}athematics {S}oftware ({V}ersion x.y.z)}.
\newblock The Sage Development Team, YYYY.
\newblock {\tt http://www.sagemath.org}.

\bibitem[{Thu}88]{Th}
William~P. {Thurston}.
\newblock {On the geometry and dynamics of diffeomorphisms of surfaces}.
\newblock {\em {Bull. Am. Math. Soc., New Ser.}}, 19(2):417--431, 1988.
\newblock \href {https://doi.org/10.1090/S0273-0979-1988-15685-6}
  {\path{doi:10.1090/S0273-0979-1988-15685-6}}.

\bibitem[{Vee}82]{Veech}
William~A. {Veech}.
\newblock {Gauss measures for transformations on the space of interval exchange
  maps}.
\newblock {\em {Ann. Math. (2)}}, 115:201--242, 1982.
\newblock \href {https://doi.org/10.2307/1971391} {\path{doi:10.2307/1971391}}.

\bibitem[{Vee}89]{Veech:construction}
William~A. {Veech}.
\newblock {Teichm\"uller curves in moduli space, Eisenstein series and an
  application to triangular billiards}.
\newblock {\em {Invent. Math.}}, 97(3):553--583, 1989.
\newblock \href {https://doi.org/10.1007/BF01388890}
  {\path{doi:10.1007/BF01388890}}.

\end{thebibliography}

\end{document}